\documentclass[a4paper,10pt,leqno]{amsart}
        \title{Splitting the relative assembly map, Nil-terms and  involutions}

\author{Wolfgang L\"uck}
\email{wolfgang.lueck@him.uni-bonn.de}
       \urladdr{http://www.math.uni-muenster.de/u/lueck}
\address{Mathematisches Institut der Universit\"at Bonn\\
                Endenicher Allee 60\\
                53115 Bonn, Germany}
\author{Wolfgang Steimle}
\email{steimle@math.uni-leipzig.de}
	  \urladdr{http://www.math.uni-leipzig.de/~steimle}  
\address{Mathematisches Institut\\
				Universit\"at Leipzig\\
				Postfach 10 09 20,
				D-04009 Leipzig, Germany}

       \date{September, 2015}
  

\usepackage[active]{srcltx}
\usepackage{pdfsync}
\usepackage{color}
\usepackage{hyperref}
\usepackage{calc}
\usepackage{enumerate,amssymb}
\usepackage{graphicx}
\usepackage[arrow,curve,matrix,tips,2cell]{xy}
  \SelectTips{eu}{10} \UseTips
  \UseAllTwocells
\usepackage{tikz}
\usetikzlibrary{decorations.pathmorphing,snakes}

\DeclareMathAlphabet{\matheurm}{U}{eur}{m}{n}

\newcommand{\addcat}{\matheurm{Add\text{-}Cat}}

\newcommand{\Zcat}{\matheurm{\IZ\text{-}Cat}}

\newcommand{\Groupoids}{\matheurm{Groupoids}}

\newcommand{\Or}{\matheurm{Or}}
\newcommand{\Spaces}{\matheurm{Spaces}}
\newcommand{\Spectra}{\matheurm{Spectra}}
\newcommand{\Sets}{\matheurm{Sets}}

\DeclareMathOperator{\asmb}{asmb}

\DeclareMathOperator{\cent}{cent}

\DeclareMathOperator{\cyl}{cyl}

\DeclareMathOperator{\ev}{ev}

\DeclareMathOperator{\gen}{gen}

\DeclareMathOperator{\id}{id}

\DeclareMathOperator{\ind}{ind}

\DeclareMathOperator{\map}{map}
\DeclareMathOperator{\mor}{mor}

\DeclareMathOperator{\op}{op}

\DeclareMathOperator{\pr}{pr}

\DeclareMathOperator{\tors}{tors}

\newcommand{\Fin}{{\mathcal{F}\text{in}}}

\newcommand{\VCyc}{{\mathcal{VC}}}

  \newcommand{\IQ}{\mathbb{Q}}
  \newcommand{\IR}{\mathbb{R}}

  \newcommand{\IZ}{\mathbb{Z}}

  \newcommand{\cala}{\mathcal{A}}
  \newcommand{\calb}{\mathcal{B}}
  \newcommand{\calc}{\mathcal{C}}
  \newcommand{\cald}{\mathcal{D}}
  
  \newcommand{\calf}{\mathcal{F}}
  \newcommand{\calg}{\mathcal{G}}

  \newcommand{\calm}{\mathcal{M}}

  \newcommand{\cals}{\mathcal{S}}

  \newcommand{\bfa}{{\mathbf a}}
  
  \newcommand{\bfb}{{\mathbf b}}

  \newcommand{\bfE}{{\mathbf E}}

  \newcommand{\bfi}{{\mathbf i}}
  
  \newcommand{\bfj}{{\mathbf j}}
  \newcommand{\bfK}{{\mathbf K}}

  \newcommand{\bfl}{{\mathbf l}}

  \newcommand{\bfu}{{\mathbf u}}

\newcommand{\bfEunderbar}{{\underline{\mathbf E}}}
\newcommand{\bfKunderbar}{{\underline{\mathbf K}}}

\newcommand{\NK}{N\!K}

\newcommand{\bfNK}{{\mathbf N}{\mathbf K}}

\newcommand{\inv}{^{-1}}

\newcounter{commentcounter}

\theoremstyle{plain}
\newtheorem{theorem}{Theorem}[section]

\newtheorem{lemma}[theorem]{Lemma}

\newtheorem{definition}[theorem]{Definition}
\theoremstyle{definition}

\newtheorem{remark}[theorem]{Remark}
\newtheorem{notation}[theorem]{Notation}

\theoremstyle{remark}

\newtheorem*{remark*}{Remark}

\makeatletter\let\c@equation=\c@theorem\makeatother

\newcommand{\version}[1]                       
{\begin{center} last edited on #1\\
last compiled on \today\\
name of texfile: \jobname
\end{center}
}

\newcommand{\EGF}[2]{E_{#2}(#1)}

\newcommand{\FGF}[1]{#1\text{-}\matheurm{FGF}}

\newcommand{\intgf}[2]{\int_{#1} #2}

\newcommand{\OrG}[1]{\matheurm{Or}(#1)}
\newcommand{\OrGF}[2]{\matheurm{Or}_{#2}(#1)}

\newcommand{\xycomsquare}[8]                
{$$\xymatrix{#1 \ar[r]^-{#2} \ar[d]^{#4} &
    #3 \ar[d]^{#5}  \\
    #6\ar[r]^-{#7} & #8 }$$}

\begin{document}

\begin{abstract}
  We show that the relative Farrell-Jones assembly map from the family of finite
  subgroups to the family of virtually cyclic subgroups for algebraic $K$-theory
  is split injective in the setting where the coefficients are additive
  categories with group action.  This generalizes a result of Bartels for rings
  as coefficients. We give an explicit description of the relative term.  This
  enables us to show that it vanishes rationally if we take coefficients in a
  regular ring.  Moreover, it is, considered as a $\IZ[\IZ/2|$-module by the
  involution coming from taking dual modules, an extended module and in
  particular all its Tate cohomology groups vanish, provided that  the infinite virtually
  cyclic subgroups of type~I of G are orientable.  The latter condition is for
  instance satisfied for torsionfree hyperbolic groups.
\end{abstract}

\maketitle


\section*{Introduction}


\subsection{Motivation}
\label{subsec:Motivation}

The \emph{$K$-theoretic Farrell-Jones Conjecture} for a group $G$ and a ring $R$
predicts that the \emph{assembly map}
\[
\asmb_n \colon H_n^G(\underline{\underline{E}}G;\bfK_R) \to H_n^G(G/G;\bfK_R) =
K_n(RG)
\]
is an isomorphism for all $n \in \IZ$. Here $\underline{\underline{E}}G =
\EGF{G}{\VCyc}$ is the \emph{classifying space for the family $\VCyc$ of
  virtually cyclic subgroups} and $H_n^G(-;\bfK_R^G)$ is the 
\emph{$G$-homology theory} associated to a specific covariant functor $\bfK^G_R$ 
from the orbit category $\Or(G)$ to the category of spectra $\Spectra$.  It satisfies
$H_n^G(G/H;\bfK_R^G) = \pi_n(\bfK^G(G/H)) = K_n(RH)$ for any subgroup 
$H \subseteq G$ and $n \in \IZ$.  The assembly map is induced by the projection
$\underline{\underline{E}}G \to G/G$.  The original source for the Farrell-Jones
Conjecture is the paper by 
Farrell-Jones~\cite[1.6 on page~257 and~1.7 on page~262]{Farrell-Jones(1993a)}. 
More information about the Farrell-Jones
Conjecture and the classifying spaces for families can be found for instance in
the survey articles~\cite{Lueck(2005s)} and~\cite{Lueck-Reich(2005)}.

Let $\underline{E}G = \EGF{G}{\Fin}$ be the \emph{classifying space for the
  family $\Fin$ of finite subgroups}, sometimes also called the
\emph{classifying space for proper $G$-actions}. The 
$G$-map $\underline{E}G \to \underline{\underline{E}}G$, which is unique up to $G$-homotopy, induces a so-called
\emph{relative assembly map}
\[
\overline{\asmb}_n \colon H_n^G(\underline{E}G;\bfK_R)  
\to H_n^G(\underline{\underline{E}}G;\bfK_R).
\]
The main result of a paper by Bartels~\cite[Theorem~1.3]{Bartels(2003b)} says that
$\overline{\asmb}_n$ is split injective for all $n \in \IZ$.

In this paper we improve on this
result in two different directions: First we generalize from the context of rings $R$ to
the context of additive categories $\cala$ with $G$-action. This improvement allows to
consider twisted group rings and involutions twisted by an orientation homomorphism $G \to
\{\pm 1\}$; moreover one obtains better inheritance properties and gets fibered versions
for free.

Secondly, we give an explicit description of the relative term in terms of so-called
NK-spectra. This becomes relevant for instance in the study of the involution on the
cokernel of the relative assembly map induced by an involution of $\cala$. In more detail,
we prove:


\subsection{Splitting the relative assembly map}
\label{subsec:Splitting_the_relative_assembly_map}

Our main splitting result is

\begin{theorem}[Splitting the $K$-theoretic assembly map from $\Fin$ to $\VCyc$]
\label{the:Splitting_the_K-theoretic_assembly_map_from_Fin_to_VCyc}
Let $G$ be a group and let $\cala$ be an additive $G$-category. Let $n$ be any integer.

Then the  relative $K$-theoretic assembly map
\[
\overline{\asmb}_n \colon H_n^G\bigl(\underline{E}G;\bfK_{\cala}^G\bigr) 
\to  H_n^G\bigl(\underline{\underline{E}}G;\bfK_{\cala}^G\bigr)
\]
is split injective. In particular we obtain a natural splitting
\[
H_n^G\bigl(\underline{\underline{E}}G;\bfK_{\cala}^G\bigr)  
\xrightarrow{\cong}  
H_n^G\bigl(\underline{E}G;\bfK_{\cala}^G\bigr) 
\oplus H_n^G\bigl(\underline{E}G \to \underline{\underline{E}}G;\bfK^G_{\cala}\bigr).
\]
Moreover, there exists an $\Or(G)$-spectrum $\bfNK^G_\cala$ and a natural isomorphism
\[H_n^G\bigl(\underline{E}G \to \EGF{G}{\VCyc_I};\bfNK^G_{\cala}\bigr) 
\xrightarrow{\cong} 
H_n^G\bigl(\underline{E}G \to \underline{\underline{E}}G;\bfK_{\cala}^G\bigr).
\]
\end{theorem}

Here $\EGF{G}{\VCyc_I}$ denotes the classifying space for the family of virtually cyclic subgroups of type I, see section \ref{sec:virtually_cyclic_groups}. The proof will appear in
Section~\ref{sec:Splitting_the_relative_assembly_map_and_identifying_the_relative_term}.
The point is that, instead of considering $\bfK^G_R$ for a ring $R$, we can treat the more
general setup $\bfK^G_{\cala}$ for an additive $G$-category $\cala$, as
explained in~\cite{Bartels-Lueck(2009coeff)}
and~\cite{Bartels-Reich(2007coeff)}.  (One obtains the case of a ring $R$ back
if one considers for $\cala$ the category $\FGF{R}$ of finitely generated free
$R$-modules with the trivial $G$-action. 
Notice that we tacitly always apply  idempotent completion to the additive categories before taking $K$-theory.)  
Whereas in~\cite[Theorem~1.3]{Bartels(2003b)} just a splitting is constructed,
we construct explicit $\Or(G)$-spectra $\bfNK^G_{\cala}$ and identify the
relative terms. This is crucial for the following results.


\subsection{Involutions and vanishing of Tate cohomology}
\label{subsec:Involutions_and_vanishing_of_Tate_cohomology}

We will prove in Subsection~\ref{subsec:Proof_of_Theorem_ref(the:induced_Nil_term_on_vcyc_I)}

\begin{theorem}[The relative term is
  induced] \label{the:induced_Nil_term_on_vcyc_I} Let $G$ be a group and let
  $\cala$ be an additive $G$-category with involution. Suppose that the
  virtually cyclic subgroups of type~I of $G$ are orientable (see
  Definition~\ref{def:orientable_cyclic_subgroups}).
  
  Then the $\IZ/2$-module $H_n\bigl(\underline{E}G \to
  \underline{\underline{E}}G;\bfK^G_{\cala}\bigr)$ is isomorphic to $\IZ[\IZ/2]
  \otimes_{\IZ} A$ for some $\IZ$-module $A$.
\end{theorem}

In~\cite{Farrell-Lueck-Steimle(2015)} we will be interested in the conclusion of Theorem~\ref{the:induced_Nil_term_on_vcyc_I}
that the Tate cohomology groups  
$\widehat{H}^i\left(\IZ/2,H_n\bigl(\underline{E}G \to \underline{\underline{E}}G;\bfK^G_{\cala}\bigr)\right)$ 
vanish for all $i,n \in \IZ$  if the virtually cyclic
subgroups of type of I  of $G$ are orientable. In general the Tate spectrum of the involution on the 
Whitehead spectrum plays a role in the study of automorphisms of manifolds (see e.g.~\cite[section 4]{Weiss-Williams(2000)}).


\subsection{Rational vanishing of the relative term}
\label{subsec:Rational_vanishing_of_the_relative_term}

\begin{theorem}\label{the:relativ_assembly_Fin_to_VCyc_rational_bijective_K-theory}
Let  $G$ be a group and let $R$ be a regular ring. 

Then the relative assembly map
\[
\overline{\asmb}_n\colon  H_n^G(\underline{E}G;\bfK^G_R) 
\to H_n^G(\underline{\underline{E}}G;\bfK^G_R) 
\]
is  rationally bijective for all $n \in \IZ$.
\end{theorem}

If $R = \IZ$ and $n \le -1$, then
the relative assembly map $H_n^G(\underline{E}G;\bfK^G_{\IZ})
\xrightarrow{\cong} H_n^G(\underline{\underline{E}}G;\bfK^G_{\IZ})$ is an
isomorphism by the results of~\cite{Farrell-Jones(1995)}.

Further computations of the relative term 
$H_n^G\bigl(\underline{E}G \to \EGF{G}{\VCyc_I};\bfNK^G_{\cala}\bigr) \cong
H_n^G\bigl(\underline{E}G \to \underline{\underline{E}}G;\bfK_{\cala}^G\bigr)$ 
are briefly discussed in
Section~\ref{sec:On_the_computation_of_the_relative_term}.


\subsection{A fibered case}
\label{subsec:a_fibered_case}
In Section~\ref{sec:fibered_version} we discuss a fibered situation which will be
relevant for the forthcoming paper~\cite{Farrell-Lueck-Steimle(2015)} and can
be handled by our general treatment for additive $G$-categories.


\subsection{Acknowledgement}
\label{subsec:acknowledgement}
This paper has been financially supported  by the Leibniz-Award of the first author
and the Center for Symmetry and Deformation at the University of Copenhagen.
We thank the referee for her/his careful reading and useful comments.


\section{Virtually cyclic groups}
\label{sec:virtually_cyclic_groups}

A virtually cyclic group $V$ is called \emph{of type~I} if it admits an
epimorphism to the infinite cyclic group, and \emph{of type~II} if it admits an
epimorphism onto the infinite dihedral group. The statements appearing in the
next lemma are well-known, we insert a proof for the reader's convenience.

\begin{lemma} \label{lem:types_of_virtually_cyclic_groups} Let $V$ be an
  infinite virtually cyclic group.
  \begin{enumerate}

  \item \label{lem:types_of_virtually_cyclic_groups:either_:type_I_or_II} 
     $V$ is either of type~I or of type~II;

  \item \label{lem:types_of_virtually_cyclic_groups:H_1_center_type_I} The
    following assertions are equivalent:

    \begin{enumerate}

    \item \label{lem:types_of_virtually_cyclic_groups:H_1_center_type_I:a} $V$
      is of type~I;

    \item \label{lem:types_of_virtually_cyclic_groups:H_1_center_type_I:b}
      $H_1(V)$ is infinite;

    \item \label{lem:types_of_virtually_cyclic_groups:H_1_center_type_I:c}
      $H_1(V)/\tors(V)$ is infinite cyclic;

    \item \label{lem:types_of_virtually_cyclic_groups:H_1_center_type_I:d} The
      center of $V$ is infinite;

    \end{enumerate}

  \item \label{lem:types_of_virtually_cyclic_groups:max_normal_subgroup:K_V}
    There exists a unique maximal normal finite subgroup $K_V\subseteq V$, i.e.,
    $K_V$ is a finite normal subgroup and every normal finite subgroup of $V$ is
    contained in $K_V$;

  \item \label{lem:types_of_virtually_cyclic_groups:exact_sequence} Let $Q_V :=
    V/K_V$. Then we obtain a canonical exact sequence
    \[
    1 \to K_V \xrightarrow{i_V} V \xrightarrow{p_V} Q_V \to 1.
    \]
    Moreover, $Q_V$ is infinite cyclic if and only if $V$ is of type~I and $Q_V$
    is isomorphic to the infinite dihedral group if and only if $V$ is of
    type~II;

  \item \label{lem:types_of_virtually_cyclic_groups:epimorphisms} Let $f \colon
    V \to Q$ be any epimorphism onto the infinite cyclic group or onto the
    infinite dihedral group.  Then the kernel of $f$ agrees with $K_V$;

  \item \label{lem:types_of_virtually_cyclic_groups:characteristic_exact_sequence}
    Let $\phi \colon V \to W$ be a homomorphism of infinite virtually cyclic
    groups with infinite image.  Then $\phi$ maps $K_V$ to $K_W$ and we obtain
    the following canonical commutative diagram with exact rows
    \[
    \xymatrix{1 \ar[r] & K_V \ar[r]^{i_V} \ar[d]^{\phi_K} & V \ar[r]^{p_V}
      \ar[d]^{\phi} & Q_V \ar[r] \ar[d]^{\phi_Q} & 1
      \\
      1 \ar[r] &  K_W \ar[r]^{i_W}  & W \ar[r]^{p_W}  &  Q_W \ar[r]  &  1 }
\]
with injective $\phi_Q$.
\end{enumerate}
\end{lemma}
\begin{proof}~\ref{lem:types_of_virtually_cyclic_groups:H_1_center_type_I}
If $V$ is of type~I, then we obtain epimorphisms 
\[
V \to H_1(V) \to H_1(V)/\tors(H_1(V)) \to \IZ.
\]
The kernel of $V \to\IZ$ is finite, since for an exact sequence 
$1 \to \IZ \xrightarrow{i} V \xrightarrow{q} H \to 1$ with finite $H$ the composite of 
$V \to \IZ$ with $i$ is injective and hence the restriction of $q$ to the kernel of
$V \to \IZ$ is injective. This implies that $H_1(V)$ is infinite and
$H_1(V)/\tors(H_1(V))$ is infinite cyclic. If $H_1(V)/\tors(H_1(V))$ is infinite
cyclic or if $H_1(V)$ is infinite, then $H_1(V)$ surjects onto $\IZ$ hence so does $V$. This shows
$\ref{lem:types_of_virtually_cyclic_groups:H_1_center_type_I:a}~\Longleftrightarrow~%
\ref{lem:types_of_virtually_cyclic_groups:H_1_center_type_I:b}~\Longleftrightarrow~%
\ref{lem:types_of_virtually_cyclic_groups:H_1_center_type_I:c}$.

Consider the exact sequence $1 \to \cent(V) \to V \to V/\cent(V) \to 1$, where
$\cent(V)$ is the center of $V$.  Suppose that $\cent(V)$ is infinite. Then $V/\cent(V)$
is finite and the Lyndon-Serre spectral yields
an isomorphism $\cent(V) \otimes_{\IZ} \IQ \to H_1(V;\IQ)$.
Hence $H_1(V)$ is infinite. This shows
$\ref{lem:types_of_virtually_cyclic_groups:H_1_center_type_I:d}~\Longrightarrow~%
\ref{lem:types_of_virtually_cyclic_groups:H_1_center_type_I:b}$.

Suppose that $V$ is of type~I. Choose an exact sequence $1 \to K \to V \to \IZ
\to 1$ with finite $K$.  Let $v \in V$ be an element which is mapped to a generator of $\IZ$.
Then conjugation with $v$ induces an automorphism of $K$. Since $K$ is finite,
we can find a natural number $k$ such that conjugation with $v^k$ induces the
identity on $K$. One easily checks that $v^k$ belongs to the center of $V$ and
$v$ has infinite order. This shows
$\ref{lem:types_of_virtually_cyclic_groups:H_1_center_type_I:a}~\Longrightarrow~%
\ref{lem:types_of_virtually_cyclic_groups:H_1_center_type_I:d}$ and finishes the
proof of assertion~\ref{lem:types_of_virtually_cyclic_groups:H_1_center_type_I}.
\\[2mm]~\ref{lem:types_of_virtually_cyclic_groups:max_normal_subgroup:K_V} If
$K_1$ and $K_2$ are two finite normal subgroups of $V$, then
\[
K_1 \cdot K_2 := \{v \in V \mid \exists k_1 \in K_1 \;\text{and} \;  k_2 \in K_2
\;\text{with}\;  v = k_1k_2\}
\] 
is a finite normal subgroup of $V$. Hence we are left to show that 
$V$ has only finitely many different finite normal subgroups. 

To see this, choose an exact sequence $1 \to \IZ
\xrightarrow{i} V \xrightarrow{f} H \to 1$ for some finite group $H$. The map $f$ induces 
a map from the finite normal subgroups of $V$ to the normal subgroups of $H$; 
we will show that it is an injection. Let $t
\in V$ be the image under $i$ of some generator of $\IZ$ and consider two finite
normal subgroups $K_1$ and $K_2$ of $V$ with $f(K_1) = f(K_2)$. Consider $k_1\in
K_1$. We can find $k_2 \in K_2$ and $n \in \IZ$ with $k_2 = k_1 \cdot t^n$.
Then $t^n$ belongs to the finite normal subgroup $K_1 \cdot K_2$. This implies
$n = 0$ and hence $k_1 = k_2$.  This shows $K_1 \subseteq K_2$. By symmetry we
get $K_1 = K_2$.  Since $H$ contains only finitely many subgroups, we conclude
that there are only finitely many different finite normal subgroups in $V$.  Now
assertion~\ref{lem:types_of_virtually_cyclic_groups:max_normal_subgroup:K_V}
follows.
\\[1mm]~\ref{lem:types_of_virtually_cyclic_groups:either_:type_I_or_II}
and~\ref{lem:types_of_virtually_cyclic_groups:exact_sequence} Let $V$ be an
infinite virtually cyclic group. Then $Q_V$ is an infinite virtually cyclic
subgroup which does not contain a non-trivial finite normal subgroup. There exists an
exact sequence $1 \to \IZ \xrightarrow{i} Q_V \xrightarrow{f} H \to 1$ for some
finite group $H$.  There exists a subgroup of index at most two $H' \subseteq H$ such
that the conjugation action of $H$ on $\IZ$ restricted to $H'$ is trivial. Put
$Q_V' = f^{-1}(H')$.  Then the center of $Q_V'$ contains $i(\IZ)$ and hence is
infinite. By
assertion~\ref{lem:types_of_virtually_cyclic_groups:H_1_center_type_I} we can
find an exact sequence $1 \to K \to Q_V'\xrightarrow{f} \IZ \to 1$ with finite $K$.  
The group $Q_V'$ contains a unique maximal normal finite subgroup $K'$ by
assertion~\ref{lem:types_of_virtually_cyclic_groups:max_normal_subgroup:K_V}.
This implies that $K' \subseteq Q_V'$ is characteristic. Since $Q_V'$ is a
normal subgroup of $Q_V$, $K' \subseteq Q_V$ is a normal subgroup and therefore
$K'$ is trivial. Hence $Q_V'$ contains no non-trivial finite normal
subgroup. This implies that $Q_V'$ is infinite cyclic.  Since $Q_V'$ is a normal
subgroup of index $2$ in $Q_V$ and $Q_V$ contains no non-trivial finite normal
 subgroup, $Q_V$ is infinite cyclic or $D_{\infty}$.

In particular we see that every infinite virtual cyclic group is of type~I or of
type~{II}.  It remains to show that  an infinite virtually cyclic group $V$
which is of type~II cannot be of type~I.  If $1 \to K \to V \to D_{\infty} \to
1$ is an extension with finite $K$, then we obtain from the Lyndon-Serre
spectral sequence an exact sequence $H_1(K) \otimes_{\IZ Q} \IZ \to H_1(V) \to
H_1(D_{\infty})$.  Hence $H_1(V)$ is finite, since both $H_1(D_{\infty})$ and
$H_1(K)$ are finite. We conclude from
assertion~\ref{lem:types_of_virtually_cyclic_groups:H_1_center_type_I} that $V$
is not of type~I. This finishes the proof of
assertions~\ref{lem:types_of_virtually_cyclic_groups:either_:type_I_or_II}
and~\ref{lem:types_of_virtually_cyclic_groups:exact_sequence}.
\\[1mm]~\ref{lem:types_of_virtually_cyclic_groups:epimorphisms} Since $V$ is
virtually cyclic, the kernel of $f$ is finite.  Since $Q$ does not contain a
non-trivial finite normal subgroup, every normal finite subgroup of $V$ is
contained in the kernel of $f$.  Hence $\ker(f)$ is the unique maximal finite
normal subgroup of $V$.
\\[1mm]~\ref{lem:types_of_virtually_cyclic_groups:characteristic_exact_sequence}
Since $K_W$ is finite and the image of $\phi$ is by assumption infinite, the
composite $p_W \circ \phi \colon V \to Q_W$ has infinite image.  Since $Q_W$ is
isomorphic to $\IZ$ or $D_{\infty}$, the same is true for the image of $p_W
\circ \phi \colon V \to Q_W$. By
assertion~\ref{lem:types_of_virtually_cyclic_groups:epimorphisms} the kernel of
$p_W \circ \phi \colon V \to Q_W$ is $K_V$. Hence $\phi(K_V) \subseteq K_W$ and
$\phi$ induces maps $\phi_K$ and $\phi_Q$ making the diagram appearing in
assertion~\ref{lem:types_of_virtually_cyclic_groups:characteristic_exact_sequence}
commutative.  Since the image of $p_W \circ \phi \colon V \to Q_W$ is infinite,
$\phi_Q(Q_V)$ is infinite. This implies that $\phi_Q$ is injective since both
$Q_V$ and $Q_W$ are isomorphic to $D_{\infty}$ or $\IZ$.  This finishes the
proof of Lemma~\ref{lem:types_of_virtually_cyclic_groups}.
\end{proof}


\section{Some categories attached to homogeneous spaces}
\label{sec:Some_categories_attached_to_homogeneous_spaces}

Let $G$ be a group and let  $S$ be a $G$-set, for instance a homogeneous space $G/H$.
Let $\calg^G(S)$ be the associated \emph{transport groupoid}.  Objects are the
elements in $S$.  The set of morphisms from $s_1$ to $s_2$ consists of those
elements $g \in G$ for which $gs_1 = s_2$.  Composition is given by the group
multiplication in $G$.  Obviously $\calg^G(S)$ is a connected groupoid
if $G$ acts transitively on $S$. A $G$-map $f \colon S \to T$ induces a functor
$\calg^G(f) \colon\calg^G(S) \to \calg^G(T)$ by sending an object $s \in S$ to
$f(s)$ and a morphism $g \colon s_1 \to s_2$ to the morphism $g \colon f(s_1)
\to f(s_2)$. We mention that for two objects $s_1$ and $s_2$ in $\calg^G(S)$ the
induced map $\mor_{\calg^G(S)}(s_1,s_2) \to \mor_{\calg^G(T)}(f(s_1),f(s_2))$ is
injective.

A functor $F \colon \calc_0 \to \calc_1$ of categories is called an
\emph{equivalence} if there exists a functor $F' \colon \calc_1 \to
\calc_0$ with the property that $F' \circ F$ is naturally equivalent
to the identity functor $\id_{\calc_0}$ and $F \circ F'$ is naturally
equivalent to the identity functor $\id_{\calc_1}$.  A functor $F$ is
a natural equivalence if and only if it is \emph{essentially
  surjective} (i.e., it induces a bijection on the isomorphism classes
of objects) and it is \emph{full} and \emph{faithful}, (i.e., for any
two objects $c,d$ in $\calc_0$ the induced map $\mor_{\calc_0}(c,d)
\to \mor_{\calc_1}(F(c),F(d))$ is bijective). 

Given a monoid $M$, let $\widehat{M}$ be the category with precisely one object
and $M$ as the monoid of endomorphisms of this object. For any subgroup $H$ of $G$, the inclusion
\[e(G/H)\colon \widehat{H}\to \calg^G(G/H), \quad g\mapsto (eH\xrightarrow{g} eH)\]
(where $e\in G$ is the unit element) is an equivalence of categories, 
whose inverse sends $g\colon g_1H \to g_2H$ to $g_2\inv g g_1\in G$.

Now fix an infinite virtually cyclic subgroup $V \subseteq G$ of type~I.  Then
$Q_V$ is an infinite cyclic group. Let $\gen(Q_V)$ be the set of generators.
Given a generator $\sigma \in \gen(Q_V)$, define $Q_V[\sigma]$ to be the
submonoid of $Q_V$ consisting of elements of the form $\sigma^n$ for $n
\in \IZ, n \ge 0$.  Let $V[\sigma] \subseteq V$ be the submonoid given by
$p_V^{-1}(Q_V[\sigma])$.  Let $\calg^G(G/V)[\sigma]$ be the subcategory of
$\calg^G(G/V)$ whose objects are the objects in $\calg^G(G/V)$
and whose morphisms $g \colon g_1V \to g_2V$ satisfy 
$g_2^{-1}gg_1 \in V[\sigma]$.  Notice that 
$\calg^G(G/V)[\sigma]$ is not a groupoid anymore,
but any two objects are isomorphic. Let $\calg^G(G/V)_K$ be the 
subcategory of $\calg^G(G/V)$ whose objects are the objects 
in $\calg^G(G/V)$ and whose morphisms
$g \colon g_1V \to g_2V$ satisfy $g_2^{-1}gg_1 \in K_V$.  Obviously
$\calg^G(G/V)_K$ is a connected groupoid and a  subcategory of $\calg^G(G/V)[\sigma]$.

We obtain the following
commutative diagram of categories
\begin{eqnarray}
\label{transport_evaluated_at_e}
&
\xymatrix@!C=8em{
\calg^G(G/V)[\sigma] \ar[r]_{j(G/V)[\sigma]}& 
\calg^G(G/V) 
\\
\widehat{V[\sigma]} \ar[r]_{\widehat{j_V[\sigma]}} \ar[u]^{e(G/V)[\sigma]}_{\simeq}
&
\widehat{V} \ar[u]_{e(G/V)}^{\simeq}
}
&
\end{eqnarray}
whose horizontal arrows are induced by the the obvious
inclusions and whose left vertical arrow is the restriction of $e(G/V)$ (and is also an equivalence of categories). 
The functor $e(G/V)$ also restricts to an equivalence of categories
\begin{eqnarray}
e(G/V)_K \colon \widehat{K_V} & \xrightarrow{\simeq} & \calg^G(G/V)_K.
\label{transport_K_evaluated_a_e}
\end{eqnarray}

\begin{remark*}
The relation of the categories $\widehat{K_V}$, $\widehat{V[\sigma]}$ and
$\widehat{V}$ to $\calg^G(G/V)_K$, $\calg^G(G/V)[\sigma]$ and $\calg^G(G/V)$ is
analogous to the relation of the fundamental group of a path connected space to
its fundamental groupoid.
\end{remark*}

Let $\overline{\sigma} \in V$ be any element which is mapped under the
projection $p_V \colon V \to Q_V$ to the fixed generator $\sigma$.  Right
multiplication with $\overline{\sigma}$ induces a $G$-map $R_{\sigma} \colon G/K_V
\to G/K_V, \; gK_V \mapsto g\overline{\sigma}K_V$.  One easily checks that
$R_{\sigma}$ is depends only on $\sigma$ and is independent of the choice of
$\overline{\sigma}$.  Let $\pr_V \colon G/K_V \to G/V$ be the projection. We obtain
the following commutative diagram
\begin{eqnarray}
\label{Z-action_on_calg(G/K)}
\xymatrix{
\calg^G(G/K_V) \ar[rr]^{R_{\sigma}} \ar[rd]_{\calg^G(\pr_V)} 
& &
\calg^G(G/K_V)\ar[ld]^{\calg^G(\pr_V)} 
\\
& 
\calg^G(G/V)
&
}
\end{eqnarray}


\section{Homotopy colimits of $\IZ$-linear and additive categories}
\label{sec:Homotopy_colimits_of_additive_categories}

Homotopy colimits of additive categories have been defined for instance
in~\cite[Section~5]{Bartels-Lueck(2009coeff)}. Here we review its definition and describe
some properties, first in the setting of $\IZ$-linear categories.

Recall that a $\IZ$-linear category is a category where all Hom-sets are provided with the
structure of abelian groups, such that composition is bilinear. Denote by $\Zcat$ the
category whose objects are $\IZ$-linear categories and whose morphisms are additive
functors between them. Given a collection of $\IZ$-linear categories $(\cala_i)_{i\in I}$,
their coproduct $\coprod_{i\in I} \calc_i$ in $\Zcat$ exists and has the following explicit description: Objects are
pairs $(i, X)$ where $i\in I$ and $X\in \cala_i$. The abelian group of morphisms $(i,
X)\to (j, Y)$ is non-zero only if $i=j$ in which case it is $\mor_{\cala_i}(X,Y)$.

Let $\calc$ be a small category.  Given a contravariant functor $F \colon \calc \to
\Zcat$, its \emph{homotopy colimit} (see for instance~\cite{Thomason(1979)}).
\begin{eqnarray}
  & \intgf{\calc}{F} &
  \label{int_calg_F}
\end{eqnarray}
is the $\IZ$-linear category obtained from the coproduct $\coprod_{c\in\calc} F(c)$ by adjoining morphisms
\[T_f\colon (d, f^*X) \to (c, X)\]
for each $(c, X)\in \coprod_{c\in \calc} F(c)$ and each morphism $f\colon d\to c$ in $\calc$. 
(Here we write $f^*X$ for $F(f)(X)$.) They are subject to the relations that $T_{\id}=\id$ and that all possible diagrams
\[\xymatrix{(e, g^*f^*X) \ar[r]^{T_g} \ar[rd]_{T_{f\circ g}} & (d,f^*X) \ar[d]^{T_f} && (d,f^*X) \ar[r]^{T_f}
  \ar[d]^{f^* u} & (c,X) \ar[d]^{u}
  \\
  & (c,X) && (d,f^*Y) \ar[r]^{T_f} & (c,Y) }\]
are to be commutative. 

Hence, a morphism in $\intgf{\calc}F$ from $(x,A)$ to $(y,B)$ can be uniquely written as a sum
\begin{equation}\label{eq:typical_morphisms_in_hocolim}
  \sum_{f \in \mor_{\calc}(x,y)} T_f \circ  \phi_f
\end{equation}
where $\phi_f \colon A \to f^*B$ is a morphism in $F(x)$ and all but finitely 
many of the morphisms $\phi_f$ are zero. The composition of two such
morphisms can be determined by the distributivity law and the rule
\[(T_f \circ \phi) \circ (T_g\circ \psi) = T_{f\circ g} \circ (g^*\phi\circ \psi)\] which just
follows the fact that both upper squares are commutative.

Using this description, it follows that the homotopy colimit has the following universal
property for additive functors $\intgf{\calc}{F}\to \cala$ into some other $\IZ$-linear
category $\cala$:

Suppose that we are given additive functors $j_c\colon F(c)\to \cala$, for each $c\in \calc$, 
and morphisms $S_f\colon j_d(f^*X) \to j_c(X)$ for each $X\in F(c)$ and each
$f\colon d\to c$ in $\calc$. If $S_{\id}=\id$ and all possible diagrams
\[\xymatrix{ j_c(g^*f^*X) \ar[r]^{S_g} \ar[rd]_{S_{f\circ g}} & j_d(f^*X) \ar[d]^{S_f} && j_d(f^*X) \ar[r]^{S_f}
  \ar[d]^{j_d(f^* u)} & j_c(X) \ar[d]^{j_c(u)}
  \\
  & j_c(X) && j_d(f^*Y) \ar[r]^{S_f} & j_c(Y) }\]
are commutative, then this data specifies an additive  functor $\intgf{\calc}{F}\to \cala$ by sending $T_f$ to $S_f$.

The homotopy colimit is functorial in $F$. Namely, if $S \colon F_0 \to F_1$ is a natural
transformation of contravariant functors $\calc \to \Zcat$, then it induces an additive functor
\begin{eqnarray}
  & \intgf{\calc}{S}  \colon \intgf{\calc}{F_0} \to \intgf{\calc}{F_1} &
  \label{intgf(calg)(S)}
\end{eqnarray}
of $\IZ$-linear categories. It is defined using the universal property by sending 
$F_0(c)$ to $F_1(c)\subset \intgf{\calc}{F_1}$ via $S$ and ``sending $T_f$ to $T_f$''. 
In more detail, the image of $T_f\colon (c, f^*(X))\to (d, X)$ (in $\intgf{\calc}{F_0}$) is given by
$T_f\colon (c, f^*(S(X)))\to (d, S(X))$ (in $\intgf{\calc}{F_1}$). 
Obviously we have for $S_1 \colon F_0 \to F_1$ and $S_2 \colon
F_1 \to F_2$
\begin{eqnarray}
  \left(\intgf{\calc}{S_2}\right) \circ \left(\intgf{\calc}{S_1}\right)
  & = &
  \intgf{\calc}{(S_2 \circ S_1)};
  \label{int_S_2_circint_S_1_is_int_S_2_circS_1}
  \\
  \intgf{\calc}{\id_F} & = & \id_{\intgf{\calc}{F}}.
  \label{int_id_is_id}
\end{eqnarray}

The construction is also functorial in $\calc$. Namely, let $W \colon \calc_1 \to \calc_2$
be a covariant functor.  Then we obtain a covariant functor
\begin{eqnarray}
  & W_* \colon \intgf{\calc_1}{F} \circ W  \to  \intgf{\calc_2}{F} &
  \label{W_ast}
\end{eqnarray}
of additive categories which is the identity on each $F(W(c))$ and sends ``$T_f$ to $T_{W(f)}$'', 
again interpreted appropriately. For covariant functors
$W_1 \colon \calc_1 \to \calc_2$, $W_2 \colon \calc_2 \to \calc_3$ and a contravariant
functor $F \colon \calc_3 \to \addcat$ we have
\begin{eqnarray}
(W_2)_* \circ (W_1)_* & = & (W_2 \circ W_1)_*;
\label{(W_2)_ast_circ_(W_1)_ast_is_(W_2_circ_W_1)_ast}
\\
(\id_{\calc})_* & = & \id_{\intgf{\calc}{F}}.
\label{(id_calg)_ast_is_id_int_calg_F}
\end{eqnarray}
These two constructions are compatible. Namely, given a natural transformation
$S \colon F_1 \to F_2$ of contravariant functors $\calc_2 \to \Zcat$ and a covariant
functor $W \colon \calc_1 \to \calc_2$, we get
\begin{eqnarray}
\left(\intgf{\calc_2}{S} \right) \circ W_*
& = &
W_* \circ \left(\intgf{\calc_1}{(S \circ W)} \right).
\label{compatibility_of_W_ast_and_int_S}
\end{eqnarray}

\begin{lemma}
\label{lem:(F_1)_ast_and_int_calg_S_and_equivalences}
\begin{enumerate}
\item \label{lem:(F_1)_ast_and_int_calg_S_and_equivalences:(F_1)_ast} Let $W
  \colon \cald \to \calc$ be an equivalence of categories.  Let $F \colon \calc \to \Zcat$ be a contravariant
  functor.  Then
\[
W_* \colon \intgf{\cald}{F \circ W} \to \intgf{\calc}{F}
\]
is an equivalence of  categories;

\item \label{lem:(F_1)_ast_and_int_calg_S_and_equivalences:int_calg_S}
Let $\calc$ be a  category and let $S \colon F_1 \to F_2$ be a transformation
of contravariant functors $\calc \to \Zcat$ such that for every
object $c$ in $\calc$ the functor $S(c) \colon F_0(c) \to F_1(c)$ is
an equivalence of categories. Then
\[
\intgf{\calc}{S} \colon \intgf{\calc}{F_1} \to \intgf{\calc}{F_2}
\]
is an equivalence of categories.
\end{enumerate}
\end{lemma}

The proof is an easy exercise. Note the general fact that if $F\colon \calc\to \cald$ is an additive functor between $\IZ$-linear categories such that $F$ is an equivalence between the underlying categories, then it follows
automatically that there exists an additive inverse equivalence $F'$ and two additive natural equivalences 
$F' \circ F \simeq \id_{\calc}$ and $F \circ F' \simeq \id_{\cald}$.

\begin{notation}
  If $W\colon \calc_1\to \calc$ is the inclusion of a subcategory, then the same is true
  for $W_*$. If no confusion is possible, we just write
  \[\intgf{\calc_1}F:=\intgf{\calc_1}{F\circ W}\subset \intgf{\calc}F.\]
\end{notation}

Denote by $\addcat$ the category whose objects are additive categories and whose morphisms
are given by additive functors between them. Notice that $\int_\calc F$ is not necessarily
an additive category even if all the $F(c)$ are -- the direct sum $(c, X) \oplus (d, Y)$
need not exist.  However any isomorphism $f\colon c\to d$ in $\calc$ induces an
isomorphism $T_f\colon (c, f^*Y) \to (d, Y)$ so that
\[(c, X) \oplus (d, Y) \cong (c, X) \oplus (c, f^*Y) \cong (c, X\oplus f^*Y).\] 
Hence, if
in the index category all objects are isomorphic and all the $F(c)$ are additive, then
$\int_\calc F$ is an additive category. As for additive categories $\cala, \calb$ we have
\[\mor_{\Zcat}(A,B) = \mor_{\addcat}(A,B)\]
(in both cases the morphisms are just additive functors), the universal property for
additive functors $\intgf{\calc}{F}\to \cala$ into $\IZ$-linear categories extends to a
universal property for additive functors into additive categories.

In the general case of an arbitrary indexing category, the homotopy colimit in the setting
of additive categories still exists. It is obtained by freely adjoining direct sums to the
homotopy colimit for $\IZ$-linear categories; the universal properties then holds in the
setting of ``additive categories with choice of direct sum''. We will not discuss this in
detail here since in all the cases we will consider, the indexing category has the
property that any two objects are isomorphic.
                                                                                                                                                                                                                                                                                                                                         
\begin{notation}
  If the indexing category $\calc$ has a single object and $F\colon \calc\to\Zcat$ is a
  contravariant functor, then we will write $X$ instead of $(*,X)$ for a typical element
  of the homotopy colimit. The structural morphisms in $\intgf{\calc}{F}$ thus take the
  simple form
  \[T_f\colon f^*X \to X\] for $f$ a morphism (from the single object to itself) in
  $\calc$.
\end{notation}


\section{The twisted Bass-Heller-Swan Theorem for additive categories}
\label{sec:The_twisted_Bass-Heller-Swan_Theorem_for_additive_categories}

Given an additive category $\cala$, we denote by $\bfK(\cala)$ the
non-connective $K$-theory spectrum associated to it (after idempotent completion),
see~\cite{Lueck-Steimle(2014delooping)},~\cite{Pedersen-Weibel(1989)}.
Thus we obtain a covariant functor
\begin{eqnarray}
\bfK \colon \addcat \to \Spectra.
\label{K-functor_addcat_to_spectra}
\end{eqnarray}

Let $\calb$ be an additive $\IZ$-category, i.e.~an additive category with a right
action of the infinite cyclic group. Fix a generator $\sigma$ of the
infinite cyclic group $\IZ$.  Let $\Phi \colon \calb \to \calb$ be the
automorphism of additive categories given by multiplication with $\sigma$.  Of
course one can recover the $\IZ$-action from $\Phi$.  Since $\widehat{\IZ}$ has
precisely one object, we can and will identify the set of objects of
$\int_{\widehat{\IZ}} \calb$ and $\calb$ in the sequel.  Let $i_{\calb} \colon
\calb \to \intgf{\widehat{\IZ}}{\calb}$ be the inclusion into the homotopy colimit.

The structural morphisms $T_\sigma\colon \Phi(B) \to B$ of $\intgf{\widehat\IZ}{\calb}$
assemble to a natural isomorphism $i_\calb\circ\Phi\to i_\calb$ in the following diagram:
\[
\xymatrix{\calb \ar[rr]^{\Phi} \ar[rd]_(.4){i_{\calb}}  
& &
\calb \ar[ld]^(.4){i_{\calb}} 
\\
& 
\int_{\widehat{\IZ}} \calb
&
}
\]
If we apply the non-connective $K$-theory spectrum to it, we obtain
a diagram of spectra which commutes up to preferred homotopy.
\[
\xymatrix{\bfK(\calb) \ar[rr]^{\bfK(\Phi)} \ar[rd]_(.4){\bfK(i_{\calb})}  
& &
\bfK(\calb) \ar[ld]^(.4){\bfK(i_{\calb})} 
\\
& 
\bfK\left(\int_{\widehat{\IZ}} \calb\right)
&
}
\]
It induces a map of spectra
\[
\bfa_{\calb} \colon T_{\bfK(\Phi)} \to \bfK\left(\int_{\widehat{\IZ}} \calb\right)
\]
where $T_{\bfK(\Phi)}$ is the mapping torus of the map of spectra 
$\bfK(\Phi) \colon \bfK(\calb) \to \bfK(\calb)$ which is defined as the pushout
\[
\xymatrix@!C=12em{
\bfK(\calb) \wedge \{0,1\}_+ = \bfK(\calb) \vee \bfK(\calb) \ar[r]^-{\bfK(\Phi) \vee  \id_{\bfK(\calb)}} \ar[d]
&
\bfK(\calb) \ar[d] 
\\
\bfK(\calb) \wedge [0,1]_+ \ar[r]
&
T_{\bfK(\Phi)}
}
\]
Denote by $\IZ[\sigma]$ the submonoid $\{\sigma^n \mid n \in \IZ, n \ge 0\}$ generated by $\sigma$. 
Let $j[\sigma] \colon \IZ[\sigma] \to \IZ$ be the inclusion.
Let $i_{\calb}[\sigma] \colon \calb \to \intgf{\widehat{\IZ]\sigma]}}{\calb}$ be the inclusion induced by $i_{\calb}$.
Define a functor of additive categories
\[
\ev_{\calb}[\sigma] \colon \int_{\widehat{\IZ[\sigma]}}{\calb} \to \calb
\]
extending the identity on $\calb$ by sending a morphism $T_{\sigma^n}$ to $0$ for $n>0$. 
(Of course $\sigma^0=\id$ must go to the identity.) We obtain the following diagram of spectra
\[
\xymatrix@!C= 10em{\bfK(\calb) \ar@/_{10mm}/[rr]_{\id} \ar[r]^{\bfK(i_{\calb}[\sigma])} 
&
\bfK\left(\intgf{\widehat{\IZ[\sigma]}}{\calb}\right)
\ar[r]^{\bfK(\ev_{\calb}[\sigma])} 
&
\bfK(\calb)
}\]
Define 
$\bfNK(\calb,\sigma)$ 
as the homotopy fiber of the  map 
$\bfK(\ev_{\calb}[\sigma]) \colon \bfK\left(\intgf{\widehat{\IZ[\sigma]}}{\calb }\right) \to \bfK(\calb) $.
Let  $\bfb_{\calb}[\sigma]$ denote the composite
\[
\bfb_{\calb}[\sigma] \colon \bfNK(\calb,\sigma) 
\to \bfK\left(\intgf{\widehat{\IZ[\sigma]}}{\calb}\right) \to \bfK\left(\int_{\widehat{\IZ}} \calb\right)
\]
of the canonical map with the inclusion. Let $\gen(\IZ)$ be the set of generators of the infinite cyclic group $\IZ$.
Put
\[
\bfNK(\calb) := \bigvee_{\sigma \in \gen(\IZ)} 
\bfNK(\calb,\sigma)
\]
and define
\[
\bfb_{\calb} := \bigvee_{\sigma \in \gen(\IZ)} \bfb_\calb[\sigma] 
\colon \bigvee_{\sigma \in \gen(\IZ)} 
\bfNK(\calb,\sigma)
\to 
\bfK\left(\intgf{\widehat{\IZ}}{\calb}\right).
\]

The proof of the following result can be found
in~\cite{Lueck-Steimle(2013twisted_BHS)}.  The case where the $\IZ$-action on
$\calb$ is trivial and one considers only $K$-groups in dimensions $n \le 1$ has
already been treated by Ranicki~\cite[Chapter~10 and~11]{Ranicki(1992a)}.  If
$R$ is a ring with an automorphism and one takes $\calb$ to be the category
$\FGF{R}$ of finitely generated free $R$-modules with the induced $\IZ$-action,
Theorem~\ref{the:Twisted_Bass-Heller-Swan_decomposition_for_additive_categories}
boils down for higher algebraic $K$-theory to the twisted
Bass-Heller-Swan-decomposition of 
Grayson~\cite[Theorem~2.1 and Theorem~2.3]{Grayson(1988)}).

\begin{theorem}[Twisted Bass-Heller-Swan decomposition for additive categories]
\label{the:Twisted_Bass-Heller-Swan_decomposition_for_additive_categories}
The map of spectra
\[
\bfa_{\calb} \vee \bfb_{\calb} \colon T_{\bfK(\Phi)}  \vee \bfNK(\calb)  \xrightarrow{\simeq} 
\bfK\left(\intgf{\widehat{\IZ}}{\calb}\right)
\]
is a weak equivalence of spectra.
\end{theorem}


\section{Some additive categories associated to an additive $G$-category}
\label{sec:Some_additive_categories_associated_to_an_additive_G-category}

Let $G$ be a group. Let $\cala$ be an additive $G$-category, i.e., an additive
category with a right $G$-operation by isomorphisms of additive categories. We
can consider $\cala$ as a contravariant functor $\widehat{G} \to \addcat$. Fix a
homogeneous $G$-space $G/H$. Let 
$\pr_{G/H} \colon \calg^G(G/H) \to \calg^G(G/G) = \widehat{G}$ 
be the projection induced by the canonical $G$-map $G/H \to
G/G$. Then we obtain a covariant functor $\calg^G(G/H) \to \addcat$ by sending
$G/H$ to $\cala \circ \pr_G$.  Let $\intgf{\calg^G(G/H)}{\cala \circ \pr_{G/H}}$
be the additive category given by the homotopy colimit (defined in~\eqref{int_calg_F}) of this functor. 
A $G$-map $f \colon G/H \to G/K$ induces a functor $\calg^G(f) \colon \calg^G(G/H)
\to \calg^G(G/K)$ which is compatible with the projections to
$\widehat{G}$. Hence it induces a functor of additive categories, see~\eqref{W_ast}
\[
\calg^G(f)_* \colon \intgf{\calg^G(G/H)}{\cala \circ \pr_{G/H}} 
\to \intgf{\calg^G(G/K)}{\cala \circ \pr_{G/K}}.
\]
Thus we obtain a covariant functor
\begin{align}
& \Or(G) \to \addcat, \quad G/H \mapsto  \intgf{\calg^G(G/H)}{\cala \circ \pr_{G/H}}. &
\label{functor_or(G)_to_addcat}
\end{align}

\begin{remark}
  Applying 
  Lemma~\ref{lem:(F_1)_ast_and_int_calg_S_and_equivalences}~\ref{lem:(F_1)_ast_and_int_calg_S_and_equivalences:(F_1)_ast}
  to the equivalence of categories $e(G/H)\colon \widehat H\to \calg^G(G/H)$, we see that
  the functor \eqref{functor_or(G)_to_addcat}, at $G/H$, takes the value
  $\intgf{\widehat H}{\cala}$ where $\cala$ carries the restricted $H$-action. The more complicated
  description is however needed for the functoriality.
\end{remark}

\begin{notation}\label{notation:simplifying}
For the sake of brevity, we will just write $\cala$ for any composite $\cala\circ \pr_{G/H}$ 
if no confusion is possible. In this notation, \eqref{functor_or(G)_to_addcat} takes the form
\[G/H\mapsto \intgf{\calg^G(G/H)} \cala.\]
\end{notation}

Let $V \subseteq G$ be an infinite virtually cyclic subgroup of type~I. 
In the sequel we abbreviate $K = K_V$ and $Q = Q_V$. Let $\pr_K \colon \calg^G(G/V)_K \to \widehat{K}$ 
the functor which sends a morphism 
$g \colon g_1V \to g_2V$ to $g_2^{-1}gg_1 \in K$.

Fixing a generator $\sigma$ of the infinite cyclic group $Q$, 
the inclusions  $\calg^G(G/V)_K\subset \calg^G(G/V)[\sigma] \subset \calg^G(G/V)$ induce inclusions
\begin{equation}\label{j(G/V)[sigma]_ast}
\intgf{\calg^G(G/V)_K}{\cala}\subset \intgf{\calg^G(G/V)[\sigma]}{\cala }  
\subset  \intgf{\calg^G(G/V)}{\cala}. 
\end{equation}

Actually the category into the middle retracts onto the the smaller one. To see this, define a retraction
\begin{equation}
\ev(G/V)[\sigma]_K \colon 
\intgf{\calg^G(G/V)[\sigma]}{\cala } \to \intgf{\calg^G(G/V)_K}{\cala}
\label{ev(G/V)[sigma]_K}
\end{equation}
as follows. It is the identity on every copy of the additive category $\cala$ inside the
homotopy colimit. Let $T_g\colon (g_1 V, g^*A)\to (g_2 V, A)$ be a structural morphism in
the homotopy colimit, where $g \colon g_1V \to g_2V$ in $\calg^G(G/V)[\sigma]$ is a
morphism in $\calg^G(G/V)[\sigma]$ (that is, $g$ is an element of $G$ satisfying 
$g_2\inv g g_1\in V[\sigma]$). If
\[g_2\inv g g_1\in K\subset V[\sigma],\]
then $g$ is by definition a morphism in $\calg^G(G/V)_K\subset \calg^G(G/V)[\sigma]$ and we may let 
\[\ev(G/V)[\sigma]_K(T_g)=T_g.\]
Otherwise we send the morphism $T_g$ to 0. This is well-defined, since for two elements $h_1,h_2 \in
V[\sigma]$ we have $h_1h_2 \in K$ if and only if both $h_1 \in K$ and 
$h_2 \in K$ hold.

Similarly the inclusion $\intgf{\widehat{K}}{\cala} 
\subset \intgf{\widehat{V[\sigma]}}{\cala}$ is split by a retraction
\[
\ev_V[\sigma] \colon 
\intgf{\widehat{V[\sigma]}}{\cala}  
\to 
 \intgf{\widehat{K}}{\cala} 
\]
defined as follows: On the copy of $\cala$ inside $\intgf{\widehat{ V[\sigma]}}\cala$, the
functor is defined to be the identity. A structural morphism $T_g\colon g^*A\to A$ is sent
to itself if $g\in K$, and to zero otherwise. One easily checks that the following diagram
commutes (where the unlabelled arrows are inclusions) and has equivalences of additive
categories as vertical maps.
\begin{eqnarray}
&
\xymatrix@!C=10em{
\intgf{\calg^G(G/V)_K}{\cala} \ar[r]  \ar@/^{10mm}/[rr]^{\id}
&
\intgf{\calg^G(G/V)[\sigma]}{\cala} \ar[r]^{\ev(G/V)[\sigma]} 
&
\intgf{\calg^G(G/V)_K}{\cala}
\\
\intgf{\widehat{K}}{\cala} \ar[r] \ar[u]_{\simeq}^{(e(G/V)_K)_*} \ar@/_{10mm}/[rr]_{\id}
&
\intgf{\widehat{V[\sigma]}}{\cala}  \ar[r]^{\ev_V[\sigma]}  \ar[u]_{\simeq}^{e(G/V)[\sigma]_*}
&
\intgf{\widehat{K}}{\cala} \ar[u]_{\simeq}^{(e(G/V)_K)_*}
}
\label{diagram_reducing_to_groups_ev}
\end{eqnarray}
We obtain from~\eqref{transport_evaluated_at_e} and
Lemma~\ref{lem:(F_1)_ast_and_int_calg_S_and_equivalences}~%
\ref{lem:(F_1)_ast_and_int_calg_S_and_equivalences:(F_1)_ast}
the following commutative diagram of additive categories with equivalences
of additive categories as vertical maps
\begin{eqnarray}
&
\xymatrix@!C=10em{
\intgf{\calg^G(G/V)[\sigma]}{\cala} \ar[r] & 
\intgf{\calg^G(G/V)}{\cala} 
\\
\intgf{\widehat{V[\sigma]}}{\cala} \ar[r] \ar[u]^{e(G/V)[\sigma]_*}_{\simeq}
&
\intgf{\widehat{V}}{\cala} \ar[u]_{e(G/V)_*}^{\simeq}
}
&
\label{diagram_of_cat_i_i[sigma]}
\end{eqnarray}
(where again the unlabelled arrows are the inclusions).

Now we abbreviate $\calb = \intgf{\widehat K}{ \cala}$. Next we define a right
$Q$-action on $\calb$ which will depend on a choice of an element $\overline{\sigma} \in
V$ such that $p_V \colon V \to Q$ sends $\overline{\sigma} $ to $\sigma$. Such an element
induces a section of the projection $G\to Q$ by which any action of $G$ induces an action
of $Q$. In short, the action of $Q$ on $\calb$ is given by the action of $G$ onto
$\cala\subset \calb$, and by the conjugation action of $Q$ on the indexing category $\widehat
K$. In more detail, the action of $\sigma\in Q$ is specified by the automorphism
\[\Phi\colon \intgf{\widehat K} \cala\to \intgf{\widehat K}\cala\]
defined as follows: A morphism $\varphi\colon A\to B$ in $\cala$ is sent to
$\overline{\sigma}^*\varphi\colon \overline{\sigma}^*A \to \overline{\sigma}^*B$, and a
structural morphism $T_g\colon g^*A\to A$ is sent to the morphism
\[T_{\overline{\sigma}\inv g \overline{\sigma}}\colon \overline\sigma ^* g^* A 
= (\overline\sigma \inv g \overline\sigma )^* \overline\sigma ^* A \to \overline\sigma^*A.\]

With this notation we obtain an additive functor
\[\Psi\colon \intgf{\widehat{Q}}\calb\to \intgf{\widehat{V}}\cala\]
defined to extend the inclusion $\calb=\intgf{\widehat{K}}\cala\to
\intgf{\widehat{V}}\cala$ and such that a structural morphism $T_\sigma\colon \Phi(A)\to A$
is sent to $T_{\overline{\sigma}}\colon \Phi(A)=\overline{\sigma}^*A\to A$. 

In more detail, a morphism in $\intgf{\widehat Q}{\calb}$ can be uniquely written as a finite sum
\[\sum_{n\in\IZ} T_{\overline{\sigma}^n} \circ \biggl(\sum_{k\in K}  T_k \circ\phi_{k,n}\biggr) 
= \sum_{n, k} T_{\overline{\sigma}^n\cdot k}\circ  \phi_{k, n}.\]
Since any element
in $V$ is uniquely a product $\overline{\sigma}^n\cdot k$ with $k\in K$, the functor $\Psi$ is
fully faithful. As it is the identity on objects, $\Psi$ is an isomorphism of
categories. It also restricts to an isomorphism of categories
\[\Psi[\sigma]\colon \intgf{\widehat{Q[\sigma]}} \calb \to \intgf{\widehat{V[\sigma]}} \cala.\]

Define a functor
\[
\ev_{\calb}[\sigma] \colon \intgf{\widehat{Q[\sigma]}}{\calb}  \to \calb
\]
as follows. It is the identity functor on $\calb$, and a non-identity structural 
morphism $T_q\colon q^*B\to B$ is sent to 0. One easily
checks using~\eqref{diagram_reducing_to_groups_ev}
and~\eqref{diagram_of_cat_i_i[sigma]} that the following diagram of additive
categories commutes (with unlabelled arrows given by inclusions) and that all vertical arrows are equivalences of additive
categories:
\begin{eqnarray}
\label{passage_to_calb}
&
\xymatrix@!C= 9em{
\intgf{\calg^G(G/V)_K}{\cala} 
&
\intgf{\calg^G(G/V)[\sigma]}{\cala} \ar[l]_{\ev(G/V)[\sigma]_K} \ar[r]
&
\intgf{\calg^G(G/V)}{\cala} 
\\
\intgf{\widehat{K}}{\cala} \ar[u]^{(e(G/V)_K)_*}_{\simeq}
& 
\intgf{\widehat{V[\sigma]}}{\cala} \ar[r] \ar[l]_{\ev_V[\sigma]} \ar[u]^{e(G/V)[\sigma]_*}_{\simeq}
&
\intgf{\widehat{V}}{\cala} \ar[u]^{e(G/V)_*}_{\simeq}
\\
\calb \ar[u]^{\id}_{\cong}
& \int_{\widehat{Q[\sigma]}} \calb \ar[r] \ar[l]_{\ev_{\calb}[\sigma]} \ar[u]^{\Psi[\sigma]}_{\cong}
& \int_{\widehat{Q}} \calb \ar[u]^{\Psi}_{\cong}
}
&
\end{eqnarray}

Recall from section~\ref{sec:Some_categories_attached_to_homogeneous_spaces} 
that $q_V \colon G/K \to G/V$ is the projection and that $R_\sigma$ is the 
automorphism of $\intgf{\calg^G(G/K)}{\cala}$ induced by right multiplication with $\overline \sigma$.

We have the following (not necessarily commutative) diagram of additive categories
all of whose   vertical arrows are equivalences of additive categories and the unlabelled arrows are the inclusions.
\begin{eqnarray}
\label{mapping_torus}
\xymatrix@!C=7em{
\intgf{\calg^G(G/K)}{\cala}  \ar[rd]_{\calg^G(\pr_V)_*}\ar[rr]^(.7){R_{\sigma}} 
& 
& 
\intgf{\calg^G(G/K)}{\cala}  \ar[ld]^{\calg^G(\pr_V)_*}
\\
& 
\intgf{\calg^G(G/V)}{\cala}  
&
\\
\intgf{\widehat{K}}{\cala} \ar[rd]  \ar[uu]^{(e(G/V)_K)_*}_{\simeq} \ar[rr]^(.7){\Phi}|{\quad } 
&
&
\intgf{\widehat{K}}{\cala} \ar[ld] \ar[uu]_{(e(G/V)_K)_*}^{\simeq}
\\
&
\intgf{\widehat{V}}{\cala} \ar[uu]^(.7){e(G/V)_*}_(.7){\simeq}
&
\\
\calb \ar[rd]_{i_{\calb}} \ar[rr]^(.7){\Phi}|{\quad }  \ar[uu]_{\cong}^{\id}
&
&
\calb \ar[ld]^{i_{\calb}} \ar[uu]^{\cong}_{\id}
\\
&
\intgf{\widehat{Q}}{\calb} \ar[uu]^(.7){\Psi}_(.7){\simeq}
&
}
\end{eqnarray}

The lowest triangle commutes up to a preferred natural isomorphism $T \colon i_{\calb}
\circ \Phi \xrightarrow{\cong} i_{\calb}$ which is part of the structural data of the
homotopy colimit. We equip the middle triangle with the natural isomorphism $\Psi \circ
T$. Explicitly it is just given by the structural morphisms $T_{\overline{\sigma}}\colon \overline{\sigma}^*A\to A$.

The three squares ranging from the middle to the
lower level commute and the two natural equivalences above are compatible with
these squares.  The top triangle commutes. The back upper square
commutes up to a preferred natural isomorphism $S \colon (e(G/V)_K)_* \circ \Phi
\xrightarrow{\cong} R_{\sigma} \circ (e(G/V)_K)_*$.  It assigns to an object $A
\in \cala$, which is the same as an object in $\intgf{\widehat{K}}{\cala}$, the
structural isomorphism 
\[S(A):=T_{\overline{\sigma}}\colon (eK, \overline\sigma^*A) \to (\overline\sigma K, A).\]

The other two squares joining the upper to the middle level commute.  From the explicit
description of the natural isomorphisms it becomes apparent that the preferred natural
isomorphism for the middle triangle defined above and the preferred natural
isomorphism for the upper back square are compatible in the sense that $e(G/V)[\sigma]_*
\circ \Psi \circ T = \calg^G(\pr_V)_* \circ S$.


\section{Some $K$-theory-spectra over the orbit category}
\label{sec:Some_K-theory-spectra_over_the_orbit_category}

In this section we introduce various $K$-theory spectra. For a detailed
introduction to spaces, spectra and modules over a category and some
constructions of K-theory spectra, we refer
to~\cite{Davis-Lueck(1998)}.

Given an additive $G$-category $\cala$, we obtain a covariant $\Or(G)$-spectrum
\begin{eqnarray}
\bfK^G_{\cala} \colon \Or(G) \to \Spectra, 
\quad G/H \mapsto \bfK\left(\intgf{\calg^G(G/H)}{\cala \circ \pr_{G/H}}\right).
\label{K-theoretic_Or(G)-spectrum}
\end{eqnarray}
by the composite of the two 
functors~\eqref{K-functor_addcat_to_spectra} and~\eqref{functor_or(G)_to_addcat}. It is naturally 
equivalent to the covariant $\Or(G)$-spectrum which is
denoted in the same way and constructed
in~\cite[Definition~3.1]{Bartels-Reich(2007coeff)}.

We again adopt notation~\ref{notation:simplifying}, abbreviating an expression such as $\cala \circ \pr_{G/H}$ just by $\cala$.
Given a virtually cyclic
subgroup $V \subseteq G$, we obtain the following map of spectra induced by the functors 
$j(G/V)[\sigma]_*$ of~\eqref{j(G/V)[sigma]_ast} and 
$\ev(G/V)[\sigma]$ of~\eqref{ev(G/V)[sigma]_K}
\[
\bfK\left(\intgf{\calg^G(G/V)_K}{\cala}\right)
\xleftarrow{\bfK(\ev(G/V)[\sigma])} 
\bfK\left(\intgf{\calg^G(G/V)[\sigma]}{\cala}\right)
\xrightarrow{\bfK(j(G/V)[\sigma]_*)} 
\bfK\left(\intgf{\calg^G(G/V)}{\cala}\right). 
\]

\begin{notation} \label{not:NK_spectra} Let $\bfNK(G/V;\cala,\sigma)$ be
  the spectrum given by the homotopy fiber of
  $\bfK(\ev(G/V)[\sigma]_*) \colon \bfK\left(\intgf{\calg^G(G/V)[\sigma]}{\cala}\right) 
\to \bfK\left(\intgf{\calg^G(G/V)_K}{\cala}\right)$.
  
  Let $\bfl \colon \bfNK(G/V;\cala,\sigma) \to \bfK\left(\intgf{\calg^G(G/V)[\sigma]}{\cala}\right)$
  be the canonical map of spectra.   Define the map of spectra
  \[
  \bfj(G/V;\cala,\sigma) \colon \bfNK(G/V;\cala,\sigma) 
\to \bfK\left(\intgf{\calg^G(G/V)}{\cala}\right)
  \]
  to be the composite $\bfK(j(G/V)[\sigma]_*) \circ \bfl$.
\end{notation}

Consider a $G$-map $f \colon G/V \to G/W$, where $V$ and $W$ are virtually cyclic groups
of type~I.  It induces a functor $\calg^G(f) \colon \calg^G(G/V) \to \calg^G(G/W)$.

It induces also a bijection 
\begin{equation}\label{eq:definition_gen(f)}
\gen(f) \colon \gen(Q_V) \to \gen(Q_W)
\end{equation}
as
follows. Fix an element $g\in G$ such that $f(eV) = gW$. Then $g^{-1}Vg
\subseteq W$. The injective group homomorphism 
$c(g) \colon V \to W, \; v \mapsto g^{-1}vg$ induces an injective group 
homomorphism $Q_{c(g)} \colon Q_V \to Q_W$ by 
Lemma~\ref{lem:types_of_virtually_cyclic_groups}~\ref{lem:types_of_virtually_cyclic_groups:characteristic_exact_sequence}.  
For $\sigma \in \gen(Q_V)$ let $\gen(f)(\sigma) \in \gen(Q_W)$ be uniquely determined by the
property that $Q_{c(g)}(\sigma) = \gen(f)(\sigma)^n$ for some $n \ge1$. 
One easily checks  that this is independent of the choice of $g \in G$
with $f(eV) = gW$ since for $w \in W$ the conjugation homomorphism 
$c(w) \colon W \to W$ induces the identity on $Q_W$. Using
 Lemma~\ref{lem:types_of_virtually_cyclic_groups}~\ref{lem:types_of_virtually_cyclic_groups:characteristic_exact_sequence} it follows
that  $\calg^G(f) \colon \calg^G(G/V) \to \calg^G(G/W)$ induces functors
\begin{eqnarray*} 
\calg^G(f)[\sigma] \colon \calg^G(G/V)[\sigma] & \to & \calg^G(G/W)[\gen(f)(\sigma)];
\\
\calg^G(f)_K\colon \calg^G(G/V)_K & \to & \calg^G(G/W)_K.
\end{eqnarray*}
Hence we obtain a commutative diagram of maps of spectra
\[
\xymatrix@!C=14em{\bfK\left(\intgf{\calg^G(G/V)_K}{\cala}\right) 
\ar[r]^{\bfK((\calg^G(f)_K)_*)}
&
\bfK\left(\intgf{\calg^G(G/W)_K}{\cala}\right)
\\
\bfK\left(\intgf{\calg^G(G/V)[\sigma]}{\cala}\right) 
\ar[u]^{\bfK(\ev(G/V)[\sigma])} \ar[r]^{\bfK(\calg^G(f)[\sigma]_*)} \ar[d]_{\bfK(j(G/V)[\sigma]_*)}
&
\bfK\left(\intgf{\calg^G(G/W)[\gen(f)(\sigma)]}{\cala}\right) 
\ar[u]^{\bfK(\ev(G/W)[\gen(f)(\sigma)])} \ar[d]_{\bfK(j(G/W)[\gen(f)(\sigma)]_*)}
\\
\bfK\left(\intgf{\calg^G(G/V)}{\cala}\right) \ar[r]^{\bfK(\calg^G(f)_*)} 
&
\bfK\left(\intgf{\calg^G(G/W)}{\cala}\right) 
}
\]

Thus we obtain a map of spectra
\[
\bfNK(f;\cala,\sigma) \colon 
\bfNK\bigl(G/V;\cala,\sigma\bigr) 
\to
\bfNK\bigl(G/W;\cala,\gen(f)(\sigma)\bigr) 
\]
such that the following diagram of spectra commutes
\[
\xymatrix@!C=13em{\bfNK\bigl(G/V;\cala,\sigma\bigr) 
\ar[r]^-{\bfNK\bigl(f;\cala,\sigma\bigr)}
\ar[d]_{\bfj(G/V;\cala,\sigma)}
&
\bfNK\bigl(G/W;\cala,\gen(f)(\sigma)\bigr) 
\ar[d]^{\bfj(G/W;\cala,\gen(f)(\sigma))}
\\
\bfK\left(\intgf{\calg^G(G/V)}{\cala}\right)
\ar[r]_-{\bfK(\calg^G(f)_*)}
&
\bfK\left(\intgf{\calg^G(G/W)}{\cala}\right)
}
\]

Let $\VCyc_I$ be the family of subgroups of $G$ which consists of all
finite groups and all virtually cyclic subgroups of type I.  Let
$\OrGF{G}{\VCyc_I}$ be the full subcategory of the orbit category
$\OrG{G}$ consisting of objects $G/V$ for which $V$ belongs to
$\VCyc_I$. Define a functor
\[
\bfNK^G_{\cala} \colon \OrGF{G}{\VCyc_I} \to \Spectra
\]
as follows. It sends $G/H$ for a finite subgroup $H$ to the trivial spectrum and
$G/V$ for a virtually cyclic subgroup $V$ of type~I to 
$\bigvee_{\sigma \in   \gen(Q_V)} \bfNK\bigl(G/V;\cala,\sigma\bigr)$.  Consider a map 
$f \colon G/V \to G/W$.  If $V$ or $W$ is finite, it is sent to the trivial map.  Suppose that
both $V$ and $W$ are infinite virtually cyclic subgroups of type~I. Then it is
sent to the wedge of the two maps
\begin{eqnarray*}
\bfNK(f;\cala,\sigma_1) \colon 
\bfNK\bigl(G/V;\cala,\sigma_1\bigr)
& \to &
\bfNK\bigl(G/W;\cala,\gen(f)(\sigma_1)\bigr); 
\\
\bfNK(f;\cala,\sigma_2) \colon 
\bfNK\bigl(G/V;\cala,\sigma_2\bigr) 
& \to &
\bfNK\bigl(G/W;\cala,\gen(f)(\sigma_2)\bigr), 
\end{eqnarray*}
for $\gen(Q_V) = \{\sigma_1, \sigma_2\}$.

The restriction of the covariant $\Or(G)$-spectrum $\bfK_{\cala}^G \colon \Or(G) \to \Spectra$
to $\OrGF{G}{\VCyc_I}$ will be denoted by the same symbol
\[
\bfK_{\cala}^G \colon \OrGF{G}{\VCyc_I} \to \Spectra.
\]
The wedge of the maps
$\bfj(G/V;\cala,\sigma_1)$ and $\bfj(G/V;\cala,\sigma_2)$ for $V$ a virtually cyclic subgroup of $G$ 
of type~I yields a map of spectra $\bfNK^G_{\cala}(G/V) \to \bfK^G_{\cala}(G/V)$.
Thus we obtain a transformation of functors from
$\OrGF{G}{\VCyc_I}$ to $\Spectra$
\begin{eqnarray}
& \bfb_{\cala}^G \colon \bfNK^G_{\cala} \to \bfK^G_{\cala}.&
\label{bG_A}
\end{eqnarray}


\section{Splitting the relative assembly map and identifying the relative term}
\label{sec:Splitting_the_relative_assembly_map_and_identifying_the_relative_term}

Let $X$ be a $G$-space. It defines a contravariant $\Or(G)$-space $O^G(X)$, i.e.,
a contravariant functor from $\Or(G)$ to the category of spaces, by sending
$G/H$ to the $H$-fixed point set $\map_G(G/H,X) = X^H$. Let $O^G(X)_+$ be the
pointed $\Or(G)$-space, where $O^G(X)_+(G/H)$ is obtained from $O^G(X)(G/H)$ by
adding an extra base point. If $f \colon X \to Y$ is a $G$-map, we obtain a natural
transformation $O^G(f)_+ \colon O^G(X)_+ \to O^G(Y)_+$.

Let $\bfE$ be a covariant $\Or(G)$-spectrum, i.e., a covariant functor from
$\Or(G)$ to the category of spectra.  Fix a $G$-space $Z$. Define the covariant
$\Or(G)$-spectrum
\[
\bfE_Z \colon \Or(G)\to \Spectra
\]
as follows. It sends an object $G/H$ to the spectrum 
$O^G(G/H \times Z)_+ \wedge_{\Or(G)} \bfE$, where
$\wedge_{\Or(G)}$ is the wedge product of a pointed spaces and a spectrum
over a category (see~\cite[Section~1]{Davis-Lueck(1998)}, where 
$\wedge_{\Or(G)}$ is denoted by $\otimes_{\Or(G)}$). The obvious identification 
of $O^G(G/H)_+(??) \wedge_{\Or(G)} \bfE(??)$ with $\bfE(G/H)$ and the projection 
$G/H \times Z \to G/H$ yields a natural transformation of covariant 
functors $\Or(G) \to \Spectra$
\begin{eqnarray}
& \bfa \colon \bfE_Z \to \bfE. &
\label{def_of_bfa}
\end{eqnarray}

\begin{lemma} \label{lem:bfE_and_bfEunderbar}
Given a $G$-space $X$, there exists an isomorphism of spectra
\[
 \bfu^G(X) \colon  O^G(X \times Z)_+ \wedge_{\Or(G)} \bfE
\xrightarrow{\cong}
O^G(X)_+  \wedge_{\Or(G)} \bfE_Z,
\]
which is natural in $X$ and $Z$.
\end{lemma}
\begin{proof}
  The smash product $\wedge_{\Or(G)}$ is associative, i.e., there is a natural
  isomorphism of spectra
\begin{multline*}
\left(O^G(X)_+(?) \wedge_{\Or(G)} O^G(?? \times Z)_+(?) \right) \wedge_{\Or(G)} \bfE(??)
\\
\xrightarrow{\cong}
O^G(X)_+(?) \wedge_{\Or(G)} \left(O^G(?? \times Z)_+(?)  \wedge_{\Or(G)} \bfE(??)\right).
\end{multline*}
There is a natural isomorphism of covariant $\Or(G)$-spaces
\[O^G(X \times Z)_+
\xrightarrow{\cong} O^G(X)_+(?) \wedge_{\Or(G)} O^G(? \times Z)_+
\]
which evaluated at $G/H$ sends $\alpha \colon G/H \to X \times Z$ to
$\bigl(\pr_1 \circ \alpha\bigr) \wedge \bigl(\id_{G/H} \times (\pr_2 \circ
\alpha)\bigr)$ if $\pr_i$ is the projection onto the $i$-th factor of $X \times
Z$. The inverse evaluated at $G/H$ sends 
$\left(\beta_1 \colon G/K \to X\right) \wedge \left(\beta_2 \colon G/H 
\to G/K \times Z\right)$ to $(\beta_1 \times \id_{Z}) \circ \beta_2$.  
The composite of these two isomorphisms yield the
desired isomorphism $\bfu^G(X)$.
\end{proof}

If $\calf$ is a family of subgroups of the group $G$, e.g., $\VCyc_I$ or the
family $\Fin$ of finite subgroups, then we denote by $\EGF{G}{\calf}$ the
classifying space of $\calf$. (For a survey on these spaces we refer
for instance to~\cite{Lueck(2005s)}.) Let $\underline{E}G$ denote the classifying space
for proper $G$-actions, or in other words, a model for $\EGF{G}{\Fin}$.  If we
restrict a covariant $\Or(G)$ spectrum $\bfE$ to $\OrGF{G}{\VCyc_I}$, we will
denote it by the same symbol $\bfE$ and analogously for $O^G(X)$.

\begin{lemma}\label{lem:isotropy_in_calf_and_wedge}
Let $\calf$ be a family of subgroups. Let $X$ be a $G$-$CW$-complex whose
isotropy groups belong to $\calf$. Let $\bfE$ be a covariant $\Or(G)$-spectrum.
Then there is a natural homeomorphism of spectra
\[
O^G(X)_+ \wedge_{\OrGF{G}{\calf}} \bfE \xrightarrow{\cong} O^G(X)_+ \wedge_{\Or(G)} \bfE.
\]
\end{lemma}
\begin{proof}
  Let $I \colon \OrGF{G}{\calf} \to \Or(G)$ be the inclusion.  The claim follows
  from the adjunction of induction $I_*$ and restriction $I^*$,
  see~\cite[Lemma~1.9]{Davis-Lueck(1998)}, and the fact that for the
  $\Or(G)$-space $O^G(X)$ the canonical map $I_*I^* O^G(X) \to O^G(X)$ is a
  homeomorphism of $\Or(G)$-spaces.
\end{proof}

In the sequel we will abbreviate $\bfE_{\underline{E}G}$ by $\bfEunderbar$.

\begin{lemma} \label{lem:E_VCyc_underbarK_versus_Efin_K} Let $\bfE$ be a
  covariant $\Or(G)$-spectrum. Let $f \colon \underline{E}G \to
  \EGF{G}{\VCyc_I}$ be a $G$-map.  (It is unique up to $G$-homotopy.) Then there
  is an up to homotopy commutative diagram of spectra whose upper horizontal
  map is a weak equivalence
\[
\xymatrix@!C=8em{
O^G(\EGF{G}{\VCyc_I}) \wedge_{\OrGF{G}{\VCyc_I}} \bfEunderbar
\ar[dr]_-{\id \wedge_{{\OrGF{G}{\VCyc_I}} } \bfa\quad \quad } \ar[rr]^-{\simeq}
& &
O^G(\underline{E}G)) \wedge_{\OrGF{G}{\VCyc_I}} \bfE
\ar[dl]^-{\quad\quad O^G(f) \wedge_{\OrGF{G}{\VCyc_I}}\id} 
\\
&
O^G(\EGF{G}{\VCyc_I}) \wedge_{\OrGF{G}{\VCyc_I}} \bfE  
&
}
\]
\end{lemma}
\begin{proof}
From Lemma~\ref{lem:bfE_and_bfEunderbar} we obtain a commutative diagram
with an isomorphism as horizontal map
\[
\xymatrix@!C=9em{O^G(\EGF{G}{\VCyc_I}) \wedge_{\OrGF{G}{\VCyc_I}} \bfEunderbar
\ar[rd]_-{\id \wedge_{\OrGF{G}{\VCyc_I}} \bfa \quad \quad} \ar[rr]^{\cong}
& &
O^G(\EGF{G}{\VCyc_I} \times \underline{E}G) \wedge_{\OrGF{G}{\VCyc_I}} \bfE 
\ar[ld]^-{\quad \quad O^G(\pr_1) \wedge_{\OrGF{G}{\VCyc_I}} \id}
\\
& 
O^G(\EGF{G}{\VCyc_I}) \wedge_{\OrGF{G}{\VCyc_I}} \bfE 
&
}
\]
where $\pr_1 \colon \EGF{G}{\VCyc_I} \times \underline{E}G) \to
\EGF{G}{\VCyc_I}$ is the obvious projection. The projection $\pr_2 \colon
\EGF{G}{\VCyc_I} \times \underline{E}G \to \underline{E}G$ is a $G$-homotopy
equivalence and its composite with $f \colon \underline{E}G \to
\EGF{G}{\VCyc_I}$ is $G$-homotopic to $\pr_1$.  Hence the following diagram of
spectra commutes up to $G$-homotopy and has a weak equivalence as horizontal
map.
\[\xymatrix@!C=9em{O^G(\EGF{G}{\VCyc_I} \times \underline{E}G) \wedge_{\OrGF{G}{\VCyc_I}} \bfE 
\ar[rr]^-{O^G(\pr_2) \wedge_{\OrGF{G}{\VCyc_I}} \id}_-{\simeq} \ar[dr]_{O^G(\pr_1) \wedge_{\OrGF{G}{\VCyc_I}} \id\quad \quad}
& &
O^G(\underline{E}G)) \wedge_{\OrGF{G}{\VCyc_I}} \bfE \ar[ld]^{\quad \quad O^G(f) \wedge_{\OrGF{G}{\VCyc_I}} \id} 
\\
&
O^G(\EGF{G}{\VCyc_I}) \wedge_{\OrGF{G}{\VCyc_I}} \bfE
&
}
\]
Putting these two diagrams together, finishes the proof of 
Lemma~\ref{lem:E_VCyc_underbarK_versus_Efin_K}
\end{proof}

If $\bfE$ is the functor $\bfK^G_{\cala}$ defined in~\eqref{K-theoretic_Or(G)-spectrum} 
and $Z = \underline{E}G$, we will write 
$\bfKunderbar^G_\cala$ for $\underline{\bfE} = \bfE_{\underline{E}G}$.

\begin{lemma} \label{bfa_wedge_bfb_weak_equiv}
Let $H$ be a finite group or an infinite virtually cyclic group of type~I.
Then the map of spectra (see~\eqref{bG_A}  and~\eqref{def_of_bfa})
\[
\bfa(G/H) \vee \bfb(G/H) \colon \bfKunderbar^G_{\cala}(G/H) \vee \bfNK^G_{\cala}(G/H)  
\to \bfK^G_{\cala}(G/H)
\]
is a weak equivalence.
\end{lemma}
\begin{proof}
Given an infinite cyclic subgroup $V \subseteq G$ of type~I, we  next 
construct the following up to homotopy commutative diagram of spectra
whose vertical arrows are all weak homotopy equivalences for $K = K_V$ and $Q = Q_V$.
Let $i_V \colon V \to G$ be the inclusion and $p_V \colon V \to Q_V :=V/K_V$ be the projection.
\[
\xymatrix@!C= 22em{
\bfKunderbar^G_{\cala}(G/V) \ar[d]^{\id}_{(1)} \ar@/^{20mm}/[rdddd]^{\bfa(G/V)}
&  
\\
O^G(G/V \times \underline{E}G)_+ \wedge_{\Or(G)} \bfK^G_{\cala}
\ar[d]^{\cong}_{(2)}
&
\\
O^G((i_V)_*(i_V)^*\underline{E}G)_+ \wedge_{\Or(G)} \bfK^G_{\cala}
\ar[d]^{\cong}_{(3)}
&
\\
O^V((i_V)^*\underline{E}G)_+ \wedge_{\Or(V)} (i_V)^*\bfK^G_{\cala}
\ar[d]^{\simeq}_{(4)}
&
\\
O^V((p_V)^*EQ_V)_+ \wedge_{\Or(V)} (i_V)^*\bfK^G_{\cala}
\ar[d]^{\cong}_{(5)}
& \bfK_{\cala}^G(G/V)
\\
O^{Q_V}(EQ_V)_+ \wedge_{\Or(Q_V)} (p_V)_! (i_V)^*\bfK^G_{\cala}
\ar[d]^{\cong}_{(6)}
&
\\
(EQ_V)_+ \wedge_{Q_V} (p_V)_! (i_V)^*\bfK^G_{\cala}(Q_V/1)
\ar[d]^{\cong}_{(7)}
&
\\
(EQ_V)_+ \wedge_{Q_V} \bfK^G_{\cala}(G/K)
\ar[d]^{\cong}_{(8)}
&
\\
T_{\bfK(R_{\sigma}) \colon \bfK^G_{\cala}(G/K) \to \bfK^G_{\cala}(G/K)}
\ar@/_{20mm}/[ruuuu]_{\bfa'(G/V)}
&
}
\]
We first explain the vertical arrow starting at the top.  The first one is the
identity by definition. The second one comes from the $G$-homeomorphism 
$G/V \times \underline{E}G \xrightarrow{\cong} (i_V)_*(i_V)^*\underline{E}G = G
\times_V \underline{E}G$ sending $(gV,x)$ to $(g,g^{-1}x)$. The third one comes
from the adjunction of induction $(i_V)_*$ and restriction $i_V^*$,
see~\cite[Lemma~1.9]{Davis-Lueck(1998)}. The fourth one comes from the fact that
$p_V^*EQ$ and $i_V^*\underline{E}G$ are both models for $\underline{E}V$ and
hence are $V$-homotopy equivalent.  The fifth one comes from the adjunction of
restriction $p_V^*$ with coinduction ${(p_V)_!}$,
see~\cite[Lemma~1.9]{Davis-Lueck(1998)}. The sixth one comes from the fact that
$EQ$ is a free $Q$-$CW$-complex and Lemma~\ref{lem:isotropy_in_calf_and_wedge}
applied to the family consisting of one subgroup, namely the trivial
subgroup. The seventh one comes from the identification $(p_V)_!
(i_V)^*\bfK^G_{\cala}(Q_V/1) = (i_V)^*\bfK^G_{\cala}(V/K) = \bfK^G_{\cala}(G/K)$.
The last one comes from the obvious homeomorphism if we use for $EQ_V$ the
standard model with $\IR$ as underlying $Q_V = \IZ$-space.  The arrow
$\bfa'(G/V)$ is induced by the upper triangle in~\eqref{mapping_torus}, which
commutes (strictly).  One easily checks that the diagram above commutes.

Here is a short explanation on the diagram above.  The map $\bfa(G/V)$ is basically given
by the projection $G/V\times \underbar EG \to G/V$. Following the equivalences (1) through
(5), this corresponds to projecting $EQ_V$ to a point. On the domain of equivalence (8),
this corresponds to projecting $EQ_V$ to a point and to take the inclusion-induced map
$\bfK^G_\cala(G/K)\to \bfK^G_\cala(G/V)$ on the other factor. But this is precisely the
definition of the map $\bfa'(G/V)$.

>From the diagram~\eqref{mapping_torus} (including the preferred equivalences and
the fact that a natural isomorphism of functors induces a preferred homotopy  after applying the
$K$-theory spectrum) we obtain the following diagram of
spectra which commutes up to homotopy and has weak homotopy equivalences as
vertical arrows.
\[
\xymatrix@!C=12em{T_{\bfK(R_{\sigma}) \colon \bfK^G_{\cala}(G/K) \to \bfK^G_{\cala}(G/K)} 
\ar[r]^-{\bfa'(G/V)}
&  
\bfK_{\cala}^G(G/V)
\\
T_{\bfK(\Phi) \colon \bfK\left(\intgf{\widehat{K}}{\cala}\right) \to \bfK\left(\intgf{\widehat{K}}{\cala}\right)} \ar[u]^{\simeq}
\ar[r]^-{\bfa''(G/V)}
&
\bfK\left(\intgf{\widehat{V}}{\cala}\right) \ar[u]^{\simeq}
\\
T_{\bfK(\phi) \colon \bfK(\calb) \to \bfK(\calb)} \ar[u]^{\id}
\ar[r]^-{\bfa_{\calb}}
&
\bfK\left(\intgf{\widehat{Q}}{\calb}\right) \ar[u]^{\cong}
}
\]
We obtain from the diagram~\eqref{passage_to_calb} the following commutative
diagram of spectra with weak homotopy equivalences as vertical arrows.
\[
\xymatrix@!C=10em{\bfNK^G_{\cala}(G/V)  
\ar[r]^{\bfb(G/V)}
& 
\bfK^G_{\cala}(G/V)
\\
\bfNK(\calb) \ar[r]^{\bfb_{\calb}} \ar[u]^{\simeq}
&
\bfK(\intgf{\widehat{Q_V}}{\calb}) \ar[u]^{\simeq}
}
\]
We conclude from the three diagrams of spectra above that
\[
\bfa(G/V) \vee \bfb(G/V) \colon \bfKunderbar^G_{\cala}(G/V) \vee \bfNK^G_{\cala}(G/V)  
\to \bfK^G_{\cala}(G/V)
\]
is a weak homotopy of spectra if and only if 
\[
\bfa_{\calb} \vee \bfb_{\calb} \colon T_{\bfK(\phi) \colon \bfK(\calb) \to \bfK(\calb)} \vee \bfNK(\calb) 
\to \bfK\left(\intgf{\widehat{Q_V}}{\calb}\right) 
\]
is a weak homotopy equivalence. Since this is just the assertion of
Theorem~\ref{the:Twisted_Bass-Heller-Swan_decomposition_for_additive_categories},
the claim of Lemma~\ref{bfa_wedge_bfb_weak_equiv} follows in the case, where $H$
is an infinite virtually cyclic group of type~I.

It remains to consider the case, where $H$ is finite. Then  $\bfNK^G_{\cala}(G/V)$ is by definition
the trivial spectrum. Hence it remains to show for a finite subgroup $H$ of $G$ that
$\bfa(G/H) \colon \bfKunderbar^G_{\cala}(G/H) 
\to \bfK^G_{\cala}(G/H)$ is a weak homotopy equivalence. This follows from the fact
that the projection $G/H \times \underline{E}G \to G/H$ is a $G$-homotopy equivalence for finite $H$.
\end{proof}

Recall that any covariant $\Or(G)$-spectrum $\bfE$ determines a $G$-homology theory $H_*^G(-;\bfE)$
satisfying $H_n^G(G/H;\bfE) = \pi_n(\bfE(G/H))$, namely put (see~\cite{Davis-Lueck(1998)})
\begin{eqnarray}
H_*^G(X;\bfE) & := & \pi_*\bigl(O^G(X) \wedge_{\Or(G)} \bfE\bigr).
\label{H_G(X;bfE)}
\end{eqnarray}

In the sequel we often follow the convention in the literature to abbreviate
$\underline{\underline{E}}G := \EGF{G}{\VCyc}$ for the family $\VCyc$ of
virtually cyclic subgroups.  Recall that for two families of subgroups $\calf_1$
and $\calf_2$ with $\calf_1 \subseteq \calf_2$ there is up to $G$-homotopy one $G$-map 
$f \colon \EGF{G}{\calf_1} \to \EGF{G}{\calf_2}$. We will define 
$H_n(\EGF{G}{\calf_1} \to \EGF{G}{\calf_2};\bfK^G_R):= H_n(\cyl(f),\EGF{G}{\calf_1};\bfK^G_R)$, 
where $(\cyl(f),\EGF{G}{\calf_1})$ is the $G$-pair coming from the mapping cylinder of $f$.

Notice that $\bfNK^G_{\cala}$ is defined only over $\OrGF{G}{VCyc_I}$. It
can be extended to a spectrum over $\Or(G)$ by applying the coinduction functor
(see~\cite [Definition~1.8]{Davis-Lueck(1998)})
associated to the inclusion $\OrGF{G}{\VCyc_I} \to \Or(G)$ so that the $G$-homology theory
$H_n^G(-;\bfNK^G_{\cala})$ makes sense for all pairs $(X,A)$ of $G$-$CW$-complexes.
Moreover, $H_n^G(X;\bfNK^G_{\cala})$ can be identified with 
$\pi_n\bigl(O^G(X) \wedge_{\OrGF{G}{\VCyc_I}} \bfNK^G_{\cala}\bigr)$ 
for all  $G$-$CW$-complexes $X$.

The remainder of this section is devoted to the proof of 
Theorem~\ref{the:Splitting_the_K-theoretic_assembly_map_from_Fin_to_VCyc}.
Its proof will need the following result  taken 
from~\cite[Remark~1.6]{Davis-Quinn-Reich(2011)}.

\begin{theorem}[Passage from $\VCyc_I$ to $\VCyc$ in $K$-theory]
\label{the:Passage_from_Vcyc_I_to_Vcyc_K-theory}
The relative assembly map
\[
H_n^G\bigl(\EGF{G}{\VCyc_I};\bfK_{\cala}^G\bigr) 
\xrightarrow{\cong} H_n^G\bigl(\underline{\underline{E}}G;\bfK_{\cala}^G\bigr)
\]
is bijective for all $n \in \IZ$.
\end{theorem}
Hence we only have to deal in the proof of Theorem~\ref{the:Splitting_the_K-theoretic_assembly_map_from_Fin_to_VCyc}
with the passage from $\Fin$ to $\VCyc_I$.

\begin{proof}[Proof of Theorem~\ref{the:Splitting_the_K-theoretic_assembly_map_from_Fin_to_VCyc}]
From  Lemma~\ref{bfa_wedge_bfb_weak_equiv} and~\cite[Lemma~4.6]{Davis-Lueck(1998)},
we obtain a weak equivalence of spectra
\begin{multline*}
\id \wedge_{\OrGF{G}{\VCyc_I}} (\bfa\vee \bfb) \colon 
O^G(\EGF{G}{\VCyc_I}) \wedge_{\OrGF{G}{\VCyc_I}}  (\bfKunderbar^G_{\cala} \vee \bfNK^G_{\cala}) 
\\
\to   O^G(\EGF{G}{\VCyc_I}) \wedge_{\OrGF{G}{\VCyc_I}}  \bfK^G_{\cala}\bigr).
\end{multline*}
Hence we obtain a weak equivalence of spectra
\begin{multline*}
\bigl(\id \wedge_{\OrGF{G}{\VCyc_I}} \bfa\bigr) \vee \bigl(\id \wedge_{\OrGF{G}{\VCyc_I}} \bfb\bigr) \colon
\\
\bigl( O^G(\EGF{G}{\VCyc_I}) \wedge_{\OrGF{G}{\VCyc_I}}  \bfKunderbar^G_{\cala}\bigr) \vee
\bigl( O^G(\EGF{G}{\VCyc_I}) \wedge_{\OrGF{G}{\VCyc_I}}  \bfNK^G_{\cala}\bigr)
\\
\to O^G(\EGF{G}{\VCyc_I}) \wedge_{\OrGF{G}{\VCyc_I}}  \bfK^G_{\cala}.
\end{multline*}
If we combine this with Lemma~\ref{lem:E_VCyc_underbarK_versus_Efin_K} we obtain
a weak equivalence of spectra
\begin{multline*}
\bigl(f \wedge_{\OrGF{G}{\VCyc_I}} \id\bigr) \vee \bigl(\id \wedge_{\OrGF{G}{\VCyc_I}} \bfb\bigr) \colon
\\
\bigl( O^G(\underline{E}G) \wedge_{\OrGF{G}{\VCyc_I}}  \bfK^G_{\cala}\bigr) \vee
\bigl( O^G(\EGF{G}{\VCyc_I}) \wedge_{\OrGF{G}{\VCyc_I}}  \bfNK^G_{\cala}\bigr)
\\
\to O^G(\EGF{G}{\VCyc_I}) \wedge_{\OrGF{G}{\VCyc_I}}  \bfK^G_{\cala}.
\end{multline*}
Using Lemma~\ref{lem:isotropy_in_calf_and_wedge} this yields a natural weak equivalence of spectra
\begin{multline*}
\bigl(f \wedge_{\Or(G)} \id\bigr) \vee \bfb' \colon
\\
\bigl( O^G(\underline{E}{G}) \wedge_{\Or(G)}  \bfK^G_{\cala}\bigr) \vee
\bigl( O^G(\EGF{G}{\VCyc_I}) \wedge_{\OrGF{G}{\VCyc_I}}  \bfNK^G_{\cala}\bigr)
\\
\to O^G(\EGF{G}{\VCyc_I}) \wedge_{\Or(G)}  \bfK^G_{\cala},
\end{multline*}
where $\bfb'$ comes from $\id \wedge_{\OrGF{G}{\VCyc_I}} \bfb$. If we take
homotopy groups, we obtain for every $n \in \IZ$ an isomorphism
\begin{multline*}
H_n^G(f;\bfK^G_{\cala}) \oplus \pi_n(\bfb')  \colon H_n^G(\underline{E}G;\bfK^G_{\cala}) 
\oplus \pi_n\bigl( O^G(\EGF{G}{\VCyc_I}) \wedge_{\OrGF{G}{\VCyc_I}}  \bfNK^G_{\cala}\bigr) 
\\
\xrightarrow{\cong} H_n\bigl(\EGF{G}{\VCyc_I};\bfK^G_{\cala}\bigr).
\end{multline*}
We have already explained above that
$H_n^G(\EGF{G}{\VCyc_I};\bfNK^G_{\cala})$ can be identified with $\pi_n\bigl(
O^G(\EGF{G}{\VCyc_I}\wedge_{\OrGF{G}{\VCyc_I}} \bfNK^G_{\cala}\bigr)$.  Since
by construction $\bfNK^G_{\cala}(G/H)$ is the trivial spectrum for finite $H$
and all isotropy groups of $\underline{E}G$ are finite, we conclude
$H_n^G(\underline{E}G;\bfNK^G_{\cala}) = 0$ for all $n \in \IZ$ from Lemma~\ref{lem:isotropy_in_calf_and_wedge}. 
We derive from the long exact sequence of $f \colon \underline{E}(G) \to \EGF{G}{\VCyc_I}$
that the canonical map
\[
H_n^G\bigl(\EGF{G}{\VCyc_I};\bfNK^G_{\cala}\bigr)
\xrightarrow{\cong} H_n^G\bigl(\underline{E}G \to \EGF{G}{\VCyc_I};\bfNK^G_{\cala}\bigr)
\]
is bijective for all $n \in \IZ$.
Hence we obtain for all $n \in \IZ$ a natural isomorphism
\begin{multline*}
H_n^G(f;\bfK^G_{\cala}) \oplus b_n
\colon H_n^G\bigl(\underline{E}G;\bfK^G_{\cala}\bigr) 
\oplus H_n^G\bigl(\underline{E}G \to \EGF{G}{\VCyc_I};\bfNK^G_{\cala}\bigr)
\\
\xrightarrow{\cong} H_n\bigl(\EGF{G}{\VCyc_I};\bfK^G_{\cala}\bigr).
\end{multline*}
We conclude from the long exact homology sequence associated 
to $f \colon \underline{E}G \to \EGF{G}{\VCyc_I}$
that the map 
\[
H_n^G(f;\bfK_{\cala}^G)  \colon H_n^G(\underline{E}G;\bfK_{\cala}^G) \to  H_n^G(\EGF{G}{\VCyc_I};\bfK_{\cala}^G)
\]
is split injective, there is a natural splitting
\[
H_n^G\bigl(\EGF{G}{\VCyc_I};\bfK_{\cala}^G\bigr)  
\xrightarrow{\cong}  
H_n^G\bigl(\underline{E}G;\bfK_{\cala}^G\bigr) \oplus H_n\bigl(\underline{E}G \to \EGF{G}{\VCyc_I};\bfK^G_{\cala}\bigr),
\]
and there exists a natural isomorphism
which is induced by the natural transformation $\bfb \colon \bfNK^G_{\cala} \to \bfK_{\cala}^G$
of spectra over $\OrGF{G}{\VCyc_I}$
\[H_n^G\bigl(\underline{E}G \to \EGF{G}{\VCyc_I};\bfNK^G_{\cala}\bigr) 
\xrightarrow{\cong} 
H_n^G\bigl(\underline{E}G \to \EGF{G}{\VCyc_I};\bfK_{\cala}^G\bigr).
\]
Now Theorem~\ref{the:Splitting_the_K-theoretic_assembly_map_from_Fin_to_VCyc} 
follows from Theorem~\ref{the:Passage_from_Vcyc_I_to_Vcyc_K-theory}.
\end{proof}


\section{Involutions and vanishing of Tate cohomology}
\label{sec:Involutions_and_vanishing_of_Tate_cohomology}

\subsection{Involutions on $K$-theory spectra}
\label{subsec:Involutions_on_K-theory_spectra}

Let $\cala = (\cala, I)$ be \emph{an additive $G$-category with involution},
i.e., an additive $G$-category $\cala$ together with a contravariant functor 
$I \colon \cala \to \cala$ satisfying $I \circ I = \id_{\cala}$ and 
$I \circ R_g = R_g \circ I$ for all $g \in G$. Examples coming from twisted group rings,
or more generally crossed product rings equipped with involutions twisted by orientation homomorphisms 
are discussed in~\cite[Section~8]{Bartels-Lueck(2009coeff)}.

In the sequel we denote for a category $\calc$ its \emph{opposite} category by
$\calc^{\op}$.  It has the same objects as $\calc$. A morphism in $\calc^{\op}$
from $x$ to $y$ is a morphism $y \to x$ in $\calc$. Obviously we can and will
identify $(\calc^{\op})^{\op} = \calc$.

Next we define a covariant functor
\begin{equation}
I(G/H) \colon \intgf{\calg^G(G/H)}{\cala} \to \left(\intgf{\calg^G(G/H)}{\cala}\right)^{\op}.
\label{def_of_I(G/H)}
\end{equation}
It is defined to extend the involution
\[\coprod_{x\in \calg^G(G/H)} I \colon \coprod_{x\in \calg^G(G/H)} \cala \to \biggl(\coprod_{x\in \calg^G(G/H)}\cala \biggr)^{\op}\]
and to send a structural morphism 
$T_g\colon (g_1H, A\cdot g) \to (g_2H, A)$ to $T_{g\inv}\colon (g_2H, I(A))\to (g_1H, I(A)\cdot g)$. 
One easily checks that $I(G/H) \circ I(G/H) = \id$.

Notice that there is a canonical identification $\bfK(\calb^{\op}) =
\bfK(\calb)$ for every additive category $\calb$.  Hence $I(G/H)$ induces a map
of spectra
\[
\bfi(G/H) = \bfK(I(G/H)) \colon \bfK\left(\intgf{\calg^G(G/H)}{\cala}\right) 
\to \bfK\left(\intgf{\calg^G(G/H)}{\cala}\right)
\]
such that $\bfi(G/H) \circ \bfi(G/H) = \id$. 
Let $\IZ/2\text{-}\Spectra$
be the category of spectra with a (strict) $\IZ/2$-operation. Thus the functor
$\bfK^G_R$ becomes a functor
\begin{eqnarray}
\bfK^G_R \colon \Or(G) \to \IZ/2\text{-}\Spectra.
\label{bfk_with_involution}
\end{eqnarray}

Consider an infinite virtually cyclic subgroup $V \subseteq G$ and a fixed
generator $\sigma \in Q_V$. The functor $I(G/V)$ of~\eqref{def_of_I(G/H)} induces
functors
\begin{eqnarray*}
I(G/H)[\sigma] \colon \intgf{\calg^G(G/H)[\sigma]}{\cala} 
& \to & 
\left(\intgf{\calg^G(G/H)[\sigma^{-1}]}{\cala}\right)^{\op};
\\
I(G/H)_K\colon \intgf{\calg^G(G/H)_K}{\cala} 
& \to & 
\left(\intgf{\calg^G(G/H)_K}{\cala}\right)^{\op}.
\end{eqnarray*}
Since $\ev(G/V)[\sigma^{-1}]_* \circ I(G/V)[\sigma]  = I(G/V)_K  \circ \ev(G/V)[\sigma]$
and $j(G/V)[\sigma^{-1}]_* \circ  I(G/V)[\sigma]  = I(G/V) \circ j(G/V)[\sigma]_*$ holds,
we obtain a commutative diagram of spectra
\[
\xymatrix@!C= 13em{
\bfK\left(\intgf{\calg^G(G/V)_K}{\cala}\right)
\ar[d]^{\bfK(I(G/V)_K)}
&
\bfK\left(\intgf{\calg^G(G/V)[\sigma]}{\cala}\right)
\ar[l]_{\bfK(\ev(G/V)[\sigma]_*)} 
\ar[r]^{\bfK(j(G/V)[\sigma]_*)} 
\ar[d]^{\bfK(I(G/V)[\sigma]}
&
\bfK\left(\intgf{\calg^G(G/V)}{\cala}\right)
\ar[d]^{\bfK(I(G/V))}
\\
\bfK\left(\intgf{\calg^G(G/V)_K}{\cala}\right)
&
\bfK\left(\intgf{\calg^G(G/V)[\sigma^{-1}]}{\cala}\right)
\ar[l]_{\bfK(\ev(G/V)[\sigma^{-1}]_*)} 
\ar[r]^{\bfK(j(G/V)[\sigma^{-1}]_*)} 
&
\bfK\left(\intgf{\calg^G(G/V)}{\cala}\right)
}
\]
Since $I(G/H)[\sigma^{-1}]  \circ I(G/H)[\sigma]  = \id$ and $I(G/H)_K \circ I(G/H)_K = \id$
holds, we obtain a $\IZ/2$-operation on  $\bfNK^G_{\cala}$ and hence a functor 
\begin{eqnarray}
\bfNK^G_{\cala} \colon \Or(G) \to \IZ/2\text{-}\Spectra,
\label{NK_as_Z/2_spectrum}
\end{eqnarray}
and we conclude:

\begin{lemma} \label{lem:bfb_compatible_with_involutions}
The transformation $\bfb \colon \bfNK^G_{\cala} \to \bfK^G_{\cala} $ of $\OrGF{G}{\VCyc_I}$-spectra
is compatible with  the $\IZ/2$-actions.
\end{lemma}

\subsection{Orientable virtually cyclic subgroups of type~I}
\label{subsec:Orientable_virtually_cyclic_subgroups_of_type_I}

\begin{definition}[Orientable virtually cyclic subgroups of type~I]
  \label{def:orientable_cyclic_subgroups}
  Given a group $G$, we say that~\emph{the infinite virtually cyclic subgroups of type~I 
of
    $G$ are orientable} if there is for every virtually cyclic subgroup $V$ of type~I a
  choice $\sigma_V$ of a generator of  the infinite cyclic group $Q_V$ with the following property: 
  Whenever $V$ and $V'$ are infinite virtually cyclic subgroups of type~I, and  $f\colon V\to V'$ 
  is an inclusion or a conjugation by some element of
  $G$, then the map $Q_f \colon Q_V \to Q_W$ sends $\sigma_V$ to a positive multiple of
  $\sigma_W$. Such a choice of elements $\{\sigma_V \mid V \in \VCyc_I\}$ 
  is called an \emph{orientation}.
\end{definition}

\begin{lemma} \label{lem:Non-orientability} Suppose that the virtually cyclic
  subgroups of type~I of $G$ are orientable. Then all infinite virtually cyclic
  subgroups of $G$ are of type~I, and the fundamental group $\IZ \rtimes \IZ$ of
  the Klein bottle is not a subgroup of $G$.
\end{lemma}

\begin{proof}
  Suppose that $G$ contains an infinite virtually cyclic subgroup $V$ of type~{II}.
  Then $Q_V$ is the infinite dihedral group.  Its commutator $[Q_V,Q_V]$ is
  infinite cyclic.  Let $W$ be the preimage of the commutator $[Q_V,Q_V]$ under
  the canonical projection $p_V \colon V \to Q_V$.  There exists an element 
  $y  \in Q_V$ such that conjugation with $y$ induces $-\id$ on $[Q_V,Q_V]$.  Obviously $W$
  is an infinite virtually cyclic group of type~I,  and the restriction of $p_V$ to $W$ is
  the canonical map $p_W \colon W \to Q_W = [Q_V,Q_V]$. Choose an element
  $x \in V$ with $p_V(x) = y$. Conjugation
  with $x$ induces an automorphism of $W$ which induces $-\id$ on $Q_W$. Hence
  the virtually cyclic subgroups of type~I of $G$ are not orientable.

The statement about the Klein bottle is obvious.
\end{proof}

For the notions of a CAT(0)-group and of a hyperbolic group we refer for
instance to~\cite{Bridson-Haefliger(1999),Ghys-Harpe(1990),Gromov(1987)}.  The
fundamental group of a closed Riemannian manifold is hyperbolic if the sectional
curvature is strictly negative, and is a CAT(0)-group if the sectional curvature
is non-positive.

\begin{lemma} \label{lem:hyperbolic_groups_orientability} Let $G$ be a
  hyperbolic group.  Then the infinite virtually cyclic subgroups of type~I of
  $G$ are orientable, if and only if all infinite virtually cyclic subgroups of
  $G$ are of type~I.
\end{lemma}

\begin{proof} The ``only if''-statement follows from Lemma~\ref{lem:Non-orientability}. 
To prove the ``if''-statement, assume that all infinite virtually
  cyclic subgroups of $G$ are of type~I.
  
  By~\cite[Example~3.6]{Lueck-Weiermann(2012)}, every  hyperbolic group
  satisfies the condition $(N\!M_{\Fin \subseteq \VCyc_I})$, i.e., every infinite
  virtually cyclic subgroup $V$ is contained in unique maximal one $V_{\max}$
  and the normalizer of $V_{\max}$ satisfies $N\!V_{\max} = V_{\max}$. Let
  $\calm$ be a complete system of representatives of the conjugacy classes of
  maximal infinite virtually cyclic subgroups.  Since by assumption $V \in \calm$ is of type~I,
 we can fix a generator $\sigma_V \in Q_V$ for each $V \in \calm$. 

  Consider any  infinite virtually cyclic
  subgroup $W$  of $G$ type $I$. Choose $g \in G$ and $V \in \calm$ 
  such that $gWg^{-1} \subseteq V$. 
  Then conjugation with $g$ induces an injection $Q_{c(g)} \colon Q_W \to Q_V$ by
  Lemma~\ref{lem:types_of_virtually_cyclic_groups}~%
\ref{lem:types_of_virtually_cyclic_groups:characteristic_exact_sequence}.
  We equip $W$ with the generator $\sigma_W \in Q_W$ 
  for which there exists an integer $n \ge 1$ with
  $Q_{c(g)}(\sigma_W) = (\sigma_V)^n$. 
  This is independent of the choice of $g$ and $V$: For every $g \in G$ and $V \in \calm$ with 
  $|gVg^{-1} \cap V| = \infty$, the condition $(N\!M_{\Fin \subseteq \VCyc_I})$
 implies that  $g$ belongs to $V$ and conjugation with an element $g \in V$ induces
  the identity on $Q_V$.
  \end{proof}

  \begin{lemma} \label{lem:CAT(0)-groups_orientability} Let $G$ be a
    CAT(0)-group.  Then the infinite virtually cyclic subgroups of type~I of $G$
    are orientable if and only if all infinite virtually cyclic subgroups of
    $G$ are of type~I and the fundamental group $\IZ \rtimes \IZ$ of the
    Klein bottle is not a subgroup of $G$.
  \end{lemma}
  \begin{proof} Because of Lemma~\ref{lem:Non-orientability} it suffices to
    construct for a CAT(0)-group an orientation for its infinite virtually
    cyclic subgroups  of type~I, provided that all infinite virtually
    cyclic subgroups of $G$ are of type~I and the fundamental group $\IZ \rtimes
    \IZ$ of the Klein bottle is not a subgroup of $G$. 

    Consider on the set of
    infinite virtually cyclic subgroups of type I of $G$ the relation $\sim$,
    where we put $V_1 \sim V_2$ if  and only if there exists an element
    $g \in G$ with $|gV_1g^{-1} \cap V_2| =
    \infty$. This is an equivalence relation since for any infinite virtually
    cyclic group $V$ and elements $v_1, v_2 \in V$ of infinite order we can find
    integers $n_1, n_2$ with $v_1^{n_1} = v_2^{n_2}$, $n_1 \not= 0$ and $n_2
    \not= 0$. Choose a complete system of representatives $\cals$ for the
    classes under $\sim$. For each element $V \in \cals$ we choose an
    orientation $\sigma_V \in Q_V$.  

    Given any infinite virtually cyclic subgroup $W \subseteq G$ of type I, we define a
    preferred generator $\sigma_W \in Q_W$ as follows. Choose $g \in G$ and 
    $V  \in \cals$ with $|gWg^{-1} \cap V| = \infty$.  Let $i_1 \colon gWg^{-1} \cap
    V \to W$ be the injection sending $v$ to $g^{-1}vg$ and $i_2 \colon gWg^{-1}
    \cap V \to V$ be the inclusion.  By
    Lemma~\ref{lem:types_of_virtually_cyclic_groups}~\ref{lem:types_of_virtually_cyclic_groups:characteristic_exact_sequence} 
    we obtain injections of infinite cyclic groups 
      $Q_{i_1} \colon Q_{gWg^{-1} \cap  V} \to Q_W$ and $Q_{i_2} \colon Q_{gWg^{-1} \cap V} \to Q_V$.  
     Equip $Q_W$ with the generator $\sigma_W$ for which there exists integers $n_1, n_2 \ge 1$
      and $\sigma \in Q_{gWg^{-1} \cap V}$ with $Q_{i_1}(\sigma) =
      (\sigma_W)^{n_1}$ and $Q_{i_2}(\sigma) = (\sigma_V)^{n_2}$.
    
      We have to show that this is well-defined. Obviously it is independent of
      the choice of $\sigma$, $n_1$ and $n_2$.  It remains to show that the
      choice of $g$ does not matter. For this purpose we have to consider the
      special case $W = V$ and have to show that the new generator $\sigma_W$
      agrees with the given one $\sigma_V$. We conclude
      from~\cite[Lemma~4.2]{Lueck(2009catevcyc)} and the argument about the
      validity of condition (C) appearing in the proof
      in~\cite[Theorem~1.1~(ii)]{Lueck(2009catevcyc)} that there exists an
      infinite cyclic subgroup $C \subseteq gVg^{-1} \cap V$ such that $g$
      belongs to the normalizer $N_GC$. It suffices to show that conjugation with
      $g$ induces the identity on $C$. Let $H \subseteq G$ be the subgroup
      generated by $g$ and $C$. We obtain a short exact sequence 
      $1 \to C \to H  \xrightarrow{\pr} H/C \to 1$, 
      where $H/C$ is a cyclic subgroup generated by
      $\pr(g)$. Suppose that $H/C$ is finite. Then $H$ is an infinite virtually
      cyclic subgroup of $G$ which must be by assumption of type~I. Since the
      center of $H$ must be infinite by
     Lemma~\ref{lem:types_of_virtually_cyclic_groups}~\ref{lem:types_of_virtually_cyclic_groups:H_1_center_type_I} 
     and hence the intersection of the center of
      $H$ with $C$ is infinite cyclic, the conjugation action of $g$ on $C$ must be
      trivial. Suppose that $H/C$ is infinite. Then $H$ is the fundamental group
      of the Klein bottle if the conjugation action of $g$ on $C$ is
      non-trivial. Since the fundamental group of the Klein bottle
      is not a subgroup of $G$ by assumption, the conjugation action of $g$ on $C$ is trivial
      also in this case.
\end{proof}

\subsection{Proof of Theorem~\ref{the:induced_Nil_term_on_vcyc_I}}
\label{subsec:Proof_of_Theorem_ref(the:induced_Nil_term_on_vcyc_I)}

Let $\Or_{\VCyc_I \setminus \Fin}(G)$ be the full subcategory of the orbit
category $\Or(G)$ consisting of those objects $G/V$ for which $V$ is an infinite  virtually
cyclic subgroup of type $I$. We obtain a functor
\[
\gen(Q_?) \colon \Or_{\VCyc_I \setminus \Fin}(G) \to \IZ/2 \text{-} \Sets
\]
sending $G/V$ to $\gen(Q_V)$, and a $G$-map $f\colon G/V\to G/W$ to $\gen(f)$ as defined in \eqref{eq:definition_gen(f)}. The $\IZ/2$-action on $\gen(Q_V)$ is given by
taking the inverse of a generator.  The condition that the virtually cyclic
subgroups of type~I of $G$ are orientable (see Definition~\ref{def:orientable_cyclic_subgroups})
is equivalent to the condition that the
functor $\gen(Q_?)$ is isomorphic to the constant functor sending $G/V$ to
$\IZ/2$. A choice of an orientation corresponds to a choice of such an isomorphism.

\begin{proof}[Proof of Theorem~\ref{the:induced_Nil_term_on_vcyc_I}]
  Because of  Theorem~\ref{the:Splitting_the_K-theoretic_assembly_map_from_Fin_to_VCyc} 
  and Lemma~\ref{lem:bfb_compatible_with_involutions} it
  suffices to show that the $\IZ[\IZ/2]$-module $H_n^G\bigl(\underline{E}G \to
  \EGF{G}{\VCyc_I};\bfNK^G_{\cala}\bigr)$ is isomorphic to 
  $\IZ[\IZ/2] \otimes_{\IZ} A$ for some $\IZ$-module $A$.

Fix an orientation $\{\sigma_V \mid V \in \VCyc_I\}$ in the sense of
Definition~\ref{def:orientable_cyclic_subgroups}.  We have the
$\OrGF{G}{\VCyc_I}$-spectrum
\[
\bfNK^{G'}_R \colon \OrGF{G}{\VCyc_I} \to \Spectra,
\]
which sends $G/V$ to the trivial spectrum if $V$ is finite and to
$\bfNK\bigl(G/V;\cala,\sigma_V\bigr)$ if $V$ is infinite virtually cyclic of type~I.  This is
well-defined by the orientability assumption. Now there is 
an obvious  natural isomorphism of functors
from $\OrGF{G}{\VCyc_I}$ to the category of $\IZ/2$-spectra 
\[
\bfNK^{G'}_{\cala} \wedge (\IZ/2)_+ \xrightarrow{\cong} \bfNK^G_{\cala}
\]
which is a weak equivalence of $\OrGF{G}{\VCyc_I}$-spectra. It induces an
$\IZ[\IZ/2]$-isomorphism
\[H_n^G\bigl(\underline{E}G \to \EGF{G}{\VCyc_I};\bfNK^{G'}_{\cala}\bigr) \otimes_{\IZ} \IZ[\IZ/2]
\xrightarrow{\cong}
H_n^G\bigl(\underline{E}G \to \EGF{G}{\VCyc_I};\bfNK^G_{\cala}\bigr).
\]
This finishes the proof of Theorem~\ref{the:induced_Nil_term_on_vcyc_I}.
\end{proof}


\section{Rational vanishing of the relative term}
\label{sec:Rational_Vanishing_of_the_relative_term}

This section is devoted to the proof of Theorem~\ref{the:relativ_assembly_Fin_to_VCyc_rational_bijective_K-theory}.

Consider the following diagram of groups, where the vertical maps are
inclusions of subgroups of finite index and the horizontal arrows are
automorphisms
\[
\xymatrix{H \ar[r]^{\phi} \ar[d]^i
&
H \ar[d]^i
\\
K \ar[r]^{\psi} 
&
K
}
\]
We obtain a commutative diagram
\begin{eqnarray}
\xymatrix{
K_n(RH_{\phi}[t]) \ar[r]^{i[t]_*} \ar[d]^{(\ev_H)_*}
&
K_n(RK_{\psi}[t]) \ar[r]^{i[t]^*} \ar[d]^{(\ev_K)_*}
&
K_n(RH_{\phi}[t]) \ar[d]^{(\ev_H)_*}
\\
K_n(RH) \ar[r]^{i_*} 
&
K_n(RK) \ar[r]^{i^*} 
&
K_n(RH)
}
\label{diagram_for_i_ast_and_iast}
\end{eqnarray}
as follows: $i_*$ and $i^*$ are the maps induced by induction and restriction with the ring
homomorphism $Ri \colon RH \to RK$; $i[t]_*$ and $i[t]^*$ are
the maps induced by induction and restriction with the ring
homomorphism $Ri[t] \colon RH_{\phi}[t] \to RK_{\psi}[t]$; $\ev_H \colon
RH_{\phi}[t] \to RH$ and $\ev_K \colon RK_{\psi}[t] \to RK$ are the ring
homomorphisms given by putting $t = 0$.

The left square is obviously well-defined and commutative. The right square is
well-defined since the restriction of $RK$ to $RH$ by $Ri$ is a finitely
generated free $RH$-module and the restriction of $RK_{\psi}[t]$ to
$RH_{\phi}[t]$ by $Ri[t]$ is a finitely generated free $RH_{\phi}$-module by the
following argument.

Put $l := [K:H]$.  Choose a subset $\{k_1, k_2, \ldots k_l\}$ of $K$ such that $K/H$
can be written  as $\{k_1H, k_2H, \ldots, k_lH\}$. The  map
\[
\alpha \colon \bigoplus_{i =1}^l RH \xrightarrow{\cong} i^*RK,
\quad   (x_1, x_2, \ldots, x_l) \mapsto \sum_{i=1}^l x_i \cdot k_i
\]
is an homomorphism of $RH$-modules and the map
\[
\beta  \colon \bigoplus_{i =1}^l RH_{\phi}[t] \xrightarrow{\cong} i[t]^* RK_{\psi}[t],
\quad   (y_1, y_2, \ldots, y_l) \mapsto \sum_{i=1}^l y_i \cdot k_i
\]
is a homomorphism of $RH_{\phi}[t]$-modules.  Obviously $\alpha$ is bijective. 
The map $\beta$ is bijective since for any integer $m$ we get
$K/H = \{\psi^m(k_1)H, \psi^m(k_2)H, \ldots \psi^m(k_i)H\}$.

To show that the right square  commutes we have to define for every finitely generated projective
$RK_{\psi}[t]$-module $P$ a natural $RH$-isomorphism
\[
T(P) \colon (\ev_H)_* i[t]^* P \xrightarrow{\cong} i^* (\ev_K)_* P.
\]
First we define $T(P)$. By the adjunction of induction and restriction it
suffices to construct a natural map $T'(P) \colon i_* (\ev_H)_* i[t]^* P \to (\ev_K)_* P$. 
Since $i \circ \ev_H = \ev_K \circ i[t]$ we have to construct a
natural map $T''(P) \colon i[t]_*i[t]^* P \to P$, since then we define $T'(P)$ to
be $(\ev_K)_*(T''(P))$.  Now define $T''(P)$ to be the adjoint of the identity
$\id \colon i[t]^* P \to i[t]^* P$. Explicitly $T(P)$ sends an element $h \otimes x$ in
$(\ev_H)_* i[t]^* P = RH \otimes_{\ev_H} i[t]^* P$ to the element $i(h) \otimes x$
in $i^* (\ev_K)_* P = RK \otimes_{\ev_K} P$.

Obviously $T(P)$ is natural in $P$ and compatible with direct sums. Hence in
order to show that $T(P)$ is bijective for all finitely generated projective
$RK_{\psi}[t]$-modules $P$, it suffices to do that for $P = RK_{\psi}[t]$. Now
the claim follows since the following diagram of $RH$-modules commutes
\[
\xymatrix@!C=12em{RH \otimes_{\ev_H} i[t]^* RK_{\psi}[t] \ar[r]^-{T(RK_{\psi}[t])} 
& i^* \bigl(RK \otimes_{\ev_K} RK_{\psi}[t]\bigr) \ar[d]^{\cong}
\\
RH \otimes_{\ev_H} \left( \bigoplus_{i =1}^l RH_{\phi}[t]  \right) \ar[u]^{\id \otimes_{\ev_H}\beta}_{\cong}
& i^* RK
\\
\bigoplus_{i =1}^l  RH \otimes_{\ev_H} RH_{\phi}[t] \ar[u]^-{\cong} \ar[r]^{\cong}
&
\bigoplus_{i =1}^l  RH \ar[u]^{\alpha}
}
\]
where the isomorphisms $\alpha$ and $\beta$ have been defined above and all
other arrows marked with $\cong$ are the obvious isomorphisms.
Recall that $\NK_n(RH,R\phi)$ is by definition  the kernel of
$(\ev_H)_* \colon  K_n(RH_{\phi}[t]) \to K_n(RH)$ and the analogous statement holds for $\NK_n(RK,R\psi)$. 

The diagram~\eqref{diagram_for_i_ast_and_iast} induces homomorphisms
\begin{eqnarray*}
i_* \colon \NK_n(RH,R\phi) & \to & \NK_n(RK,R\psi)
\\
i^* \colon \NK_n(RK,R\psi) & \to & \NK_n(RH,R\phi)
\end{eqnarray*}
Since both composites
\begin{eqnarray*}
K_n(RH_{\phi}[t]) & \xrightarrow{i[t]^* \circ i[t]_*} & K_n(RH_{\phi}[t]);
\\
K_n(RH) & \xrightarrow{i^* \circ i_*} & K_n(RH),
\end{eqnarray*}
are multiplication with $l$, we conclude

\begin{lemma} \label{lem:induction_and_restriction_on_widetilde(NK)} 
The composite 
$i^* \circ i_* \colon \NK_n(RH,R\phi) \to   \NK_n(RH,R\phi)$ 
is multiplication with $l$ for all $n \in \IZ$.
\end{lemma}

\begin{lemma} \label{lem_NIL_for_inner_automorphism} Let $\phi \colon K \to K$
  be an inner automorphisms of the group $K$.  Then there is for all $n \in \IZ$
  an isomorphism
\[
\NK_n(RK, R\phi) \xrightarrow{\cong} \NK_n(RK).
\]
\end{lemma}
\begin{proof}
Let $k$ be an element such that $\phi$ is given by conjugation with $k$.
We obtain a ring isomorphism
\[
\eta \colon RK_{R\phi}[t] \xrightarrow{\cong} RK[t], 
\quad \sum_i \lambda_i t^i \mapsto \lambda_i k^i t^i.
\]
Let $\ev_{RK,\phi} \colon RK_{\phi}[t] \to RK$ and $\ev_{RK} \colon RK[t] \to RK$ 
be the ring homomorphisms given by putting $t = 0$. Then we obtain a
commutative diagram with isomorphisms as vertical arrows
\[
\xymatrix{K_n( RK_{R\phi}[t]) \ar[r]^{\eta}_{\cong} \ar[d]^{\ev_{RK,\phi}}
&
K_n(RK[t]) \ar[d]^{\ev_{RK}}
\\
K_n(RK) \ar[r]^{\cong}_{\id}
&
K_n(RK)
}
\]
It induces the desired isomorphism 
$\NK_n(RK, R\phi) \xrightarrow{\cong} \NK_n(RK)$.
\end{proof}

\begin{remark*}
As the referee has pointed out, this results holds more generally (with identical proof) for the twisted Bass' group $NF(S,\phi)$ of any functor $F$ from rings to abelian groups and any inner ring automorphism $\phi\colon S\to S$. 
\end{remark*}

\begin{theorem}\label{the:vanishing_of_widetilde(NK)(RK,Rphi)_for_finite_K_rationally}
  Let $R$ be a regular ring. Let $K$ be a finite group of order $r$ and let
  $\phi \colon K \xrightarrow{\cong} K$ be an automorphism of order $s$.

  Then $\NK_n(RK,R\phi)[1/rs] = 0$ for all $n \in \IZ$. In particular $\NK_n(RK,R\phi) \otimes_{\IZ}
  \IQ = 0$ for all $n \in \IZ$.  
\end{theorem}
\begin{proof} Let $t$ be a generator of the cyclic group $\IZ/s$ of order $s$.
  Consider the semi-direct product $K \rtimes_{\phi} \IZ/s$. Let $i \colon K \to
  K \rtimes_{\phi} \IZ/s$ be the canonical inclusion. Let $\psi$ be the inner
  automorphism of $K \rtimes_{\phi} \IZ/s$ given by conjugation with $t$. Then
  $[K \rtimes_{\phi} \IZ/s:K] = s$ and the following diagram commutes
\[
\xymatrix{K \ar[r]^{\phi} \ar[d]_i
& K \ar[d]^i
\\
K \rtimes_{\phi} \IZ/s \ar[r]_{\psi} 
&
K \rtimes_{\phi} \IZ/s
}
\]
Lemma~\ref{lem:induction_and_restriction_on_widetilde(NK)} 
and Lemma~\ref{lem_NIL_for_inner_automorphism} yield maps
$i_* \colon \NK_n(RK,\phi) \to NK_n(R[K \rtimes_{\phi} \IZ/s])$
and
$i^* \colon NK_n(R[K \rtimes_{\phi} \IZ/s]) \to \NK_n(RK,\phi)$
such that $i^* \circ i_* = s \cdot \id$. This implies that 
$\NK_n(RK,\phi)[1/s]$ is a direct summand in 
$NK_n(R[K \rtimes_{\phi} \IZ/s])[1/s]$.
Since $R$ is regular by assumption and hence $\NK_n(R)$ vanishes for all $n \in \IZ$,
we conclude from~\cite[Theorem~A]{Hambleton-Lueck(2012)}
\[
\NK_n(R[K \rtimes_{\phi}\IZ/s])[1/rs] = 0.
\]
(For $R = \IZ$ and some related rings, 
this has already been proved by Weibel~\cite[(6.5), p.~490]{Weibel(1981)}.)
This implies
$\NK_n(RK,\phi)[1/rs] = 0$.
\end{proof}

Theorem~\ref{the:vanishing_of_widetilde(NK)(RK,Rphi)_for_finite_K_rationally}
has already been proved for $R = \IZ$
in~\cite[Theorem~5.11]{Grunewald(2008Nil)}.

Now we are ready to give the proof of Theorem~\ref{the:relativ_assembly_Fin_to_VCyc_rational_bijective_K-theory}.
\begin{proof}[Proof of Theorem~\ref{the:relativ_assembly_Fin_to_VCyc_rational_bijective_K-theory}]
Because of Theorem~\ref{the:Splitting_the_K-theoretic_assembly_map_from_Fin_to_VCyc} it suffices
to prove for all $n \in \IZ$
\[
H_n^G\bigl(\underline{E}G \to \EGF{G}{\VCyc_I};\bfNK^G_R\bigr) \otimes_{\IZ} \IQ
\xrightarrow{\cong} \{0\}.
\]
There is a spectral sequence converging to 
$H_{p+q}^G\bigl(\underline{E}G \to \EGF{G}{\VCyc_I};\bfNK^G_R\bigr)$
whose $E^2$-term is the Bredon homology 
\[
E^2_{p,q} = H_p^{\IZ\OrGF{G}{\VCyc_I}}\bigl(\underline{E}G \to \EGF{G}{\VCyc_I};\pi_q(\bfNK^G_R)\bigr)
\]
with coefficients in the covariant functor from $\OrGF{G}{\VCyc_I}$ to the category of
$\IZ$-modules coming from composing $\bfNK^G_R \colon \OrGF{G}{\VCyc_I}
\to \Spectra$ with the functor taking the $q$-homotopy group
(see~\cite[Theorem~4.7 and Theorem~7.4]{Davis-Lueck(1998)}).  Since $\IQ$ is flat over $\IZ$, it
suffices to show for all $V \in \VCyc_I$
\[\pi_q(\bfNK^G_R(G/V))\otimes_{\IZ} \IQ = 0
\]
If $V$ is finite, $\bfNK^G_R(G/V)$ is by construction the trivial
spectrum and the claim is obviously true.  If $V$ is a virtually cyclic group of
type~I, then we conclude from the diagram~\eqref{diagram_reducing_to_groups_ev}
\[
\pi_n(\bfNK^G_R(G/V)) \cong \NK_n(RK_V,R\phi) \oplus \NK_n(RK_V,R\phi^{-1}).
\]
Now the claim follows from Theorem~\ref{the:vanishing_of_widetilde(NK)(RK,Rphi)_for_finite_K_rationally}.
\end{proof}


\section{On the computation of the relative term}
\label{sec:On_the_computation_of_the_relative_term}

In this section we give some further information about the computation of the relative term
$H_n^G\bigl(\underline{E}G \to \underline{\underline{E}}G;\bfK_R^G\bigr) 
\cong H_n^G\bigl(\underline{E}G \to \EGF{G}{\VCyc};\bfNK_R^G\bigr)$.

In~\cite{Lueck-Weiermann(2012)} one can find a systematic analysis how the space
$\EGF{G}{\VCyc_I}$ is obtained from $\underline{E}G$. We say that $G$ satisfies
the condition $(M_{\Fin \subseteq \VCyc_I})$ if any virtually cyclic subgroup of
type~I is contained in a unique maximal infinite cyclic subgroup of type~I. 
We say that $G$ satisfies the condition $(N\!M_{\Fin
  \subseteq \VCyc_I})$ if it satisfies $(M_{\Fin \subseteq \VCyc_I})$ and for
any maximal virtually cyclic subgroup $V$ of type $I$ its normalizer $N_GV$
agrees with $V$. Every word hyperbolic group satisfies $(N\!M_{\Fin  \subseteq \VCyc_I})$,
see~\cite[Example~3.6]{Lueck-Weiermann(2012)}.

Suppose that $G$ satisfies $(M_{\Fin \subseteq \VCyc_I})$. Let $\calm$ be a
complete system of representatives $V$ of the conjugacy classes of maximal
virtually cyclic subgroups of type~I.  Then we conclude
from~\cite[Corollary~2.8]{Lueck-Weiermann(2012)} that there exists a $G$-pushout
of $G$-$CW$-complexes with inclusions as horizontal maps
\[
\xymatrix{\coprod_{V\in\calm} G\times_{N_GV}\underline{E}N_GV
\ar[d]^{\coprod_{V\in\calm} \id_G\times f_V} \ar[r]^-i & \underline{E}G\ar[d]^f \\
\coprod_{V\in\calm} G\times_{N_GV}\EGF{N_GV}{\VCyc_I} \ar[r] & \EGF{G}{\VCyc_I}}
\]
This yields for all $n \in \IZ$ an isomorphism using the induction structure in the sense 
of~\cite[Section~1]{Lueck(2002b)}
\[
\bigoplus_{V \in \calm}
H_n^{N_GV}\bigl(\underline{E}N_GV \to  \EGF{N_GV}{\VCyc_I};\bfK^{N_GV}_R\bigr)
\xrightarrow{\cong}
H_n^G\bigl(\underline{E}G \to \EGF{G}{\VCyc_I};\bfK_R^G\bigr).
\]
Combining this with Theorem~\ref{the:Splitting_the_K-theoretic_assembly_map_from_Fin_to_VCyc}
yields the isomorphism
\[
\bigoplus_{V \in \calm}
H_n^{N_GV}\bigl(\underline{E}N_GV \to  \EGF{N_GV}{\VCyc_I};\bfNK^{N_GV}_R\bigr)
\xrightarrow{\cong}
H_n^G\bigl(\underline{E}G \to \EGF{G}{\VCyc_I};\bfK_R^G\bigr).
\]

Suppose now that $G$ satisfies $(N\!M_{\Fin \subseteq \VCyc_I})$ and recall that, by definition, $NK^G_R(V/H)=0$ for finite $H$. Then the
isomorphism above reduces to the isomorphism
\[
\bigoplus_{V \in \calm} \pi_n\bigl(\bfNK^{V}_R(V/V)\bigr)
\xrightarrow{\cong}
H_n^G\bigl(\underline{E}G \to \EGF{G}{\VCyc_I};\bfK_R^G\bigr),
\]
and $\pi_n\bigl(\bfNK^V_R(V/V)\bigr)$ is the Nil-term
$\NK_n(RK_V,R\phi) \oplus \NK_n(RK_V;R\phi^{-1})$ appearing in the twisted
version of the Bass-Heller-Swan-decomposition of $RV$ 
(see~\cite[Theorem~2.1 and Theorem~2.3]{Grayson(1988)}) if we write 
$V \cong K_V \rtimes_{\phi} \IZ$.


\section{Fibered version}
\label{sec:fibered_version}

We illustrate in this section by an example which will be crucial in~\cite{Farrell-Lueck-Steimle(2015)}
that we do get information from our setting also in a fibered situation.

Let $p \colon X \to B$ be a map of path connected spaces. We will assume 
that it is $\pi_1$-surjective, i.e., induces an epimorphism on fundamental groups.
Suppose that $B$ admits a universal covering $q \colon \widetilde{B} \to B$.

Choose base points $x_0 \in X$, $b_0 \in B$ and $\widetilde{b_0} \in \widetilde{B}$
satisfying $p(x_0) = b_0 = q(\widetilde{b_0})$. We will abbreviate $\Gamma = \pi_1(X,x_0)$ and
$G = \pi_1(B,b_0)$. Recall that we have a free right proper $G$-action on $\widetilde{B}$
and $q$ induces a homeomorphism $\widetilde{B}/G \xrightarrow{\cong} B$. For 
a subgroup $H \subseteq G$ denote by $q(G/H) \colon \widetilde{B} \times_{G} G/H = \widetilde{B}/H \to B$
the obvious covering induced by $q$. The pullback construction yields a commutative square of spaces
\[
\xymatrix@!C= 7em{
X(G/H) \ar[r]^-{\overline{q}(G/H)} \ar[d]_{\overline{p}(G/H)}
&
X \ar[d]^{p}
\\
\widetilde{B} \times_{G} G/H \ar[r]_-{q(G/H)}
&
B
}
\]
where $\overline q(G/H)$ is again a covering. 
This yields covariant functors from the orbit category of $G$ to the category of topological spaces
\begin{eqnarray*}
\underline{B}  \colon \Or(G) & \to & \Spaces, \quad G/H \to \widetilde{B} \times_{G} G/H;
\\
\underline{X} \colon \Or(G)  & \to & \Spaces, \quad G/H \to X(G/H).
\end{eqnarray*}
The assumption that $p$ is $\pi_1$-surjective ensures that $X(G/H)$ is path connected
for all $H \subseteq G$. 

By composition with the fundamental groupoid functor we obtain a functor
\[
\underline{\Pi(X)} \colon \Or(G) \to \Groupoids, \quad G/H \mapsto \Pi(X(G/H))
\]
Let $\FGF{R}$ be the additive category whose set of objects is $\{R^n\mid n =
0,1,2,\ldots\}$ and whose morphisms are $R$-linear maps. In the sequel it will
always be equipped with the trivial $G$ or $\Gamma$-action or considered as
constant functor $\calg \to \addcat$. Consider the functor
\[
\xi \colon \Groupoids \to \Spectra, \quad \calg \mapsto \bfK\left(\intgf{\calg}{\FGF{R}}\right).
\]
The composite of the last two functors yields a functor
\[
\bfK(p):= \xi  \circ \underline{\Pi(X)}  \colon \Or(G) \to \Spectra.
\]

Associated to this functor there is a $G$-homology theory 
$H_*^G(-;\bfK(p)) := \pi_n\bigl(O^G(-) \wedge_{\Or(G)} \bfK(p)\bigr)$ 
(see~\cite{Davis-Lueck(1998)}). We will be interested in the associated assembly map 
induced by the projection 
$\underline{\underline{E}}G \to G/G$,
\begin{eqnarray}
& H_n^G\bigl(\underline{\underline{E}}G;\bfK(p)\bigr) 
\to H_n^G\bigl(G/G;\bfK(p)\bigr) \cong K_n(R\Gamma).&
\label{fibered_assembly_map}
\end{eqnarray}

The goal of this section is to identify this assembly map with the assembly map
\[H_n^G\bigl(\underline{\underline{E}}G;\bfK_\cala\bigr) 
\to H_n^G\bigl(G/G;\bfK_\cala\bigr) = K_n(R\Gamma)\]
for a suitable additive category with $G$-action $\cala$. Thus the results of 
this paper apply also in the fibered setup.

Consider the following functor
\[\underline{\calg^{\Gamma}}  \colon \Or(G)  \to  \Groupoids, \quad G/H \to \calg^{\Gamma}(G/H),\]
where we consider $G/H$ as a $\Gamma$-set by restriction along the group homomorphism $\Gamma\to G$ induced by $p$.

\begin{lemma} \label{lem:natural_equivalence_T_calg_1_to_calg_2}
There is a natural equivalence
\[
T \colon \underline{\calg^{\Gamma}}  \to \underline{\Pi(X)}
\]
of covariant functors $\Or(G)   \to  \Groupoids$.
\end{lemma}
\begin{proof}
  Given an object $G/H$ in $\Or(G)$, we have to specify an equivalence  of groupoids
  $T(G/H) \colon   \calg^{\Gamma}(G/H) \to \Pi(X(G/H))$.  
   For an object in $\calg^{\Gamma}( G/H)$
  which is given by an element $wH \in G/H$,  define
  $T(wH)$ to be the point in $X(G/H)$ which is determined by $(\widetilde{b_0},wH) \in
  \widetilde{B}\times_{G} G/H$ and $x_0 \in X$. This makes sense since
  $q(G/H)\bigl((\widetilde{b_0},wH)\bigr) = b_0 = q(x_0)$.  

   Let $\gamma \colon w_0H \to  w_1H$ be a morphism in $\calg^{\Gamma}(G/H)$.  
   Choose a loop $u_X$ in $X$ at $x_0 \in X$ which represents $\gamma$. 
   Let $u_B$ be the loop $p \circ u_X$ in $B$ at $b_0 \in B$. 
    There is precisely one path $u_{\widetilde{B}}$ in $\widetilde{B}$ which starts at
    $\widetilde{b_0}$ and satisfies $q \circ u_{\widetilde{B}} = u_B$.  Let $[u_B] \in G$ be the class of
    $u_B$, or, equivalently, the image of $\gamma$ under $\pi_1(p,x_0) \colon \Gamma \to G$.
    By definition of the right $G$-action on $\widetilde{B}$ we have
    $\widetilde{b_0} \cdot [u_B] = u_B(1)$.  Define a path $u_{\widetilde{B}/H}$ in
    $\widetilde{B}\times_{G} G/H$ from $(\widetilde{b_0},w_0H)$ to $(\widetilde{b_0},w_1H)$ by 
    $t \mapsto (u_B(t),w_0H)$. This is indeed a path ending at $(\widetilde{b_0},w_1H)$ since
    $(\widetilde{b_0} \cdot [u_B], w_0H) = (\widetilde{b_0}, [u_B] \cdot w_0H) =
    (\widetilde{b_0},w_1H)$ holds in $\widetilde{B}\times_{G} G/H$. 
    Obviously the composite of $u_{\widetilde{B}/H}$ with
    $q(G/H) \colon \widetilde{B} \times_{G} G/H \to B$ is $u_B$. Hence
    $u_{\widetilde{B}/H}$ and $u_X$ determine a path in $X(G/H)$ from $T(w_0H) \to
    T(w_1H)$ and hence  a morphism $T(w_0H) \to T(w_1H)$  in $\Pi(X(G/H))$.
    One easily checks that the homotopy class (relative to the endpoints) of $u$ depends
    only on $\gamma$.  Thus we obtain the desired functor $T(G/H) \colon
    \calg^{\Gamma}(G/H) \to \Pi(X(G/H))$.  One easily checks that they fit together
    so that we obtain a natural transformation 
    $T \colon \underline{\calg^{\Gamma}} \to \underline{\Pi(X)}$.

    At a homogeneous space $G/H$, the value of $\underline{\calg^\Gamma}$ is a
groupoid equivalent to the group $\pi_1(p,x_0)\inv(H)$ while the value of
$\underline{\Pi(X)}$ is a groupoid equivalent to the fundamental group of $X(G/H)$.
Up to this equivalence, the functor $T$, at $G/H$, is the standard identification of these two groupoids. 
Hence $T$ is a natural equivalence.
  \end{proof}

We obtain a covariant functor
\begin{eqnarray*}
\bfK(p)' \colon \Or(G) & \to & \Spectra, \quad G/H \mapsto \bfK\left(\intgf{\calg^\Gamma(G/H)}{\FGF{R}}\right).
\end{eqnarray*}
Lemma~\ref{lem:natural_equivalence_T_calg_1_to_calg_2} implies that the
following diagram commutes where the vertical arrow is the isomorphism induced by $T$.
\[
\xymatrix@!C= 20em{
H_n^G\bigl(\underline{\underline{E}}G;\bfK(p)\bigr) \ar[rd]^-{H_n^G(\pr;\bfK(p))}
&
\\
& H_n^G\bigl(G/G;\bfK(p)\bigr) = K_n(R\Gamma)
\\
H_n^G\bigl(\underline{\underline{E}}G;\bfK(p)'\bigr) \ar[ru]_-{H_n^G(\pr;\bfK(p)')}
\ar[uu]_{\cong}^{T_*} &
}
\]
Now the functor $\bfK(p)'$ is, up to natural equivalence, of the form $\bfK^G_{\cala}$ 
for some additive $G$-category, namely for
$\cala=\ind_{q\colon \Gamma \to G} \FGF{R}$, see~\cite[(11.5) and Lemma~11.6]{Bartels-Lueck(2009coeff)}. 
We conclude

\begin{lemma} \label{lem:fibered_versu_additive_category}
The assembly map~\eqref{fibered_assembly_map} 
is an isomorphism for all $n \in \IZ$, if the $K$-theoretic Farrell-Jones Conjecture for
additive categories holds for $G$. 
\end{lemma}



\begin{thebibliography}{10}

\bibitem{Bartels-Lueck(2009coeff)}
A.~Bartels and W.~L\"uck.
\newblock On twisted group rings with twisted involutions, their module
  categories and ${L}$-theory.
\newblock In {\em Cohomology of groups and algebraic $K$-theory}, volume~12 of
  {\em Advanced Lectures in Mathematics}, pages 1--55, Somerville, U.S.A.,
  2009. International Press.

\bibitem{Bartels-Reich(2007coeff)}
A.~Bartels and H.~Reich.
\newblock Coefficients for the {F}arrell-{J}ones {C}onjecture.
\newblock {\em Adv. Math.}, 209(1):337--362, 2007.

\bibitem{Bartels(2003b)}
A.~C. Bartels.
\newblock On the domain of the assembly map in algebraic {$K$}-theory.
\newblock {\em Algebr. Geom. Topol.}, 3:1037--1050 (electronic), 2003.

\bibitem{Bridson-Haefliger(1999)}
M.~R. Bridson and A.~Haefliger.
\newblock {\em Metric spaces of non-positive curvature}.
\newblock Springer-Verlag, Berlin, 1999.
\newblock Die Grundlehren der mathematischen Wissenschaften, Band 319.

\bibitem{Davis-Lueck(1998)}
J.~F. Davis and W.~L{\"u}ck.
\newblock Spaces over a category and assembly maps in isomorphism conjectures
  in ${K}$- and ${L}$-theory.
\newblock {\em $K$-Theory}, 15(3):201--252, 1998.

\bibitem{Davis-Quinn-Reich(2011)}
J.~F. Davis, F.~Quinn, and H.~Reich.
\newblock {Algebraic $K$-theory over the infinite dihedral group: a controlled
  topology approach.}
\newblock {\em J. Topol.}, 4(3):505--528, 2011.

\bibitem{Farrell-Jones(1993a)}
F.~T. Farrell and L.~E. Jones.
\newblock Isomorphism conjectures in algebraic ${K}$-theory.
\newblock {\em J. Amer. Math. Soc.}, 6(2):249--297, 1993.

\bibitem{Farrell-Jones(1995)}
F.~T. Farrell and L.~E. Jones.
\newblock The lower algebraic ${K}$-theory of virtually infinite cyclic groups.
\newblock {\em $K$-Theory}, 9(1):13--30, 1995.

\bibitem{Farrell-Lueck-Steimle(2015)}
F.~T. Farrell, W.~L\"uck, and W.~Steimle.
\newblock Approximately fibering a manifold over an aspherical one.
\newblock in preparation, 2015.

\bibitem{Ghys-Harpe(1990)}
{\'E}.~Ghys and P.~de~la Harpe, editors.
\newblock {\em Sur les groupes hyperboliques d'apr\`es {M}ikhael {G}romov}.
\newblock Birkh\"auser Boston Inc., Boston, MA, 1990.
\newblock Papers from the Swiss Seminar on Hyperbolic Groups held in Bern,
  1988.

\bibitem{Grayson(1988)}
D.~R. Grayson.
\newblock The {$K$}-theory of semilinear endomorphisms.
\newblock {\em J. Algebra}, 113(2):358--372, 1988.

\bibitem{Gromov(1987)}
M.~Gromov.
\newblock Hyperbolic groups.
\newblock In {\em Essays in group theory}, pages 75--263. Springer-Verlag, New
  York, 1987.

\bibitem{Grunewald(2008Nil)}
J.~Grunewald.
\newblock The behavior of {N}il-groups under localization and the relative
  assembly map.
\newblock {\em Topology}, 47(3):160--202, 2008.

\bibitem{Hambleton-Lueck(2012)}
I.~Hambleton and W.~L\"uck.
\newblock Induction and computation of {B}ass {N}il groups for finite groups.
\newblock {\em PAMQ}, 8(1):199--219, 2012.

\bibitem{Lueck(2002b)}
W.~L{\"u}ck.
\newblock Chern characters for proper equivariant homology theories and
  applications to ${K}$- and ${L}$-theory.
\newblock {\em J. Reine Angew. Math.}, 543:193--234, 2002.

\bibitem{Lueck(2005s)}
W.~L{\"u}ck.
\newblock Survey on classifying spaces for families of subgroups.
\newblock In {\em Infinite groups: geometric, combinatorial and dynamical
  aspects}, volume 248 of {\em Progr. Math.}, pages 269--322. Birkh\"auser,
  Basel, 2005.

\bibitem{Lueck(2009catevcyc)}
W.~L{\"u}ck.
\newblock On the classifying space of the family of virtually cyclic subgroups
  for {C}{A}{T}(0)-groups.
\newblock {\em M\"unster J. of Mathematics}, 2:201--214, 2009.

\bibitem{Lueck-Reich(2005)}
W.~L{\"u}ck and H.~Reich.
\newblock The {B}aum-{C}onnes and the {F}arrell-{J}ones conjectures in {$K$}-
  and {$L$}-theory.
\newblock In {\em Handbook of $K$-theory. Vol. 1, 2}, pages 703--842. Springer,
  Berlin, 2005.

\bibitem{Lueck-Steimle(2013twisted_BHS)}
W.~L\"uck and W.~Steimle.
\newblock A twisted {B}ass-{H}eller-{S}wan decomposition for the non-connective
  {$K$}-theory of additive categories.
\newblock Preprint, arXiv:1309.1353 [math.KT], to appear in Forum, 2013.

\bibitem{Lueck-Steimle(2014delooping)}
W.~L\"uck and W.~Steimle.
\newblock Delooping {$K$}-theory for additive categories.
\newblock In C.~Ausoni, K.~Hess, K.~Johnson, W.~L\"uck, and J.~Scherer,
  editors, {\em An Alpine Expedition through Algebraic Topology}, volume 617 of
  {\em Contemporary Mathematics}, pages 205--236. AMS, 2014.

\bibitem{Lueck-Weiermann(2012)}
W.~L\"uck and M.~Weiermann.
\newblock On the classifying space of the family of virtually cyclic subgroups.
\newblock {\em PAMQ}, 8(2):497--555, 2012.

\bibitem{Pedersen-Weibel(1989)}
E.~K. Pedersen and C.~A. Weibel.
\newblock ${K}$-theory homology of spaces.
\newblock In {\em Algebraic topology (Arcata, CA, 1986)}, pages 346--361.
  Springer-Verlag, Berlin, 1989.

\bibitem{Ranicki(1992a)}
A.~A. Ranicki.
\newblock {\em Lower ${K}$- and ${L}$-theory}.
\newblock Cambridge University Press, Cambridge, 1992.

\bibitem{Thomason(1979)}
R.~W. Thomason.
\newblock Homotopy colimits in the category of small categories.
\newblock {\em Math. Proc. Cambridge Philos. Soc.}, 85(1):91--109, 1979.

\bibitem{Weibel(1981)}
C.~A. Weibel.
\newblock Mayer-{V}ietoris sequences and module structures on {$NK\sb\ast $}.
\newblock In {\em Algebraic $K$-theory, Evanston 1980 (Proc. Conf.,
  Northwestern Univ., Evanston, Ill., 1980)}, volume 854 of {\em Lecture Notes
  in Math.}, pages 466--493. Springer, Berlin, 1981.

\bibitem{Weiss-Williams(2000)}
M.~S. Weiss and B.~Williams.
\newblock Products and duality in {W}aldhausen categories.
\newblock {\em Trans. Amer. Math. Soc.}, 352(2):689--709, 2000.

\end{thebibliography}


\end{document}